\documentclass[10pt,a4paper]{article}
\usepackage{amsmath}
\usepackage{amsfonts}
\usepackage{graphicx}
\usepackage{float}
\usepackage{caption}
\usepackage{subcaption}
\usepackage{pdflscape}
\usepackage[round]{natbib} % omit 'round' option if you prefer square brackets
\usepackage{amssymb}
\usepackage[margin={2.5cm,2.5cm}]{geometry}
\usepackage[font=small,skip=-5pt]{caption}
\usepackage{amsthm}
\usepackage{color}
\usepackage[dvipsnames]{xcolor}
\usepackage[hidelinks]{hyperref}
\usepackage{nameref}
\usepackage{dsfont}
\usepackage{enumerate}
\usepackage[inline, shortlabels]{enumitem}
\usepackage[overload]{empheq}
\usepackage{tikz}
\usepackage{multicol}
\usepackage[]{algorithm2e}
\usepackage{xr}
\usepackage{empheq}
\usepackage{url}
\usepackage{hhline}
\usepackage[title]{appendix}
\usepackage{adjustbox}
\usepackage{pgfplots}
\usepackage{tablefootnote}
\usepackage[bottom, perpage]{footmisc}

\newtheorem{theorem}{Theorem}[section]
\newtheorem{proposition}{Proposition}[section]
\newtheorem{lemma}{Lemma}[section]

\newtheorem{remark}{Remark}[section]

\newtheorem{definition}{Definition}[section]
\newtheorem{example}{Example}[section]
\newtheorem{assumption}{Assumption}[section]

\newlist{thmcases}{enumerate}{1}
\setlist[thmcases]{
  label=\textbf{\upshape Case~\arabic*.},
  leftmargin=*,
  ref={Case~\arabic*}}
  
\newlist{thmsubcases}{enumerate}{1}
\setlist[thmsubcases]{
  label=\textbf{\upshape Subcase~\arabic*.},
  leftmargin=*,
  ref={\thetheorem.\arabic*}}
  
\DeclareMathOperator{\WD}{WD}
\DeclareMathOperator{\KL}{KL}
\DeclareMathOperator{\argmax}{arg\,max}
\DeclareMathOperator{\JS}{J-S}

\providecommand{\keywords}[1]
{
  \textbf{\textbf{Key Words.  }} #1
}

\title{Distributionally robust tail bounds based on Wasserstein distance and $f$-divergence}
\author{Corina Birghila\footnotemark[2] \and Maximilian Aigner\footnotemark[3] \and Sebastian Engelke\footnotemark[4]}
\date{}

\begin{document}

\maketitle

\begin{abstract}
In this work, we provide robust bounds on the tail probabilities and the tail index of heavy-tailed distributions in the context of model misspecification. They are defined as the optimal value when computing the worst-case tail behavior over all models within some neighborhood of the reference model. The choice of the discrepancy between the models used to build this neighborhood plays a crucial role in assessing the size of the asymptotic bounds. We evaluate the robust tail behavior in ambiguity sets based on the Wasserstein distance and Csisz{\'a}r $f$-divergence and obtain explicit expressions for the corresponding asymptotic bounds. In an application to Danish fire insurance claims we compare the difference between these bounds and show the importance of the choice of discrepancy measure.

\noindent\keywords{tail index, model misspecification, distributionally robustness, Wasserstein distance, $f$-divergence}
\end{abstract}

\footnotetext[2]{Department of Statistics and Actuarial Science, University of Waterloo, 200 University Ave. W. N2L 3G1 Waterloo, ON, Canada, \texttt{ corina.birghila@uwaterloo.ca}}
\footnotetext[3]{Faculty of Business and Economics, University of Lausanne, 1015 Lausanne, Switzerland, \texttt{maximilian.aigner@unil.ch}}
\footnotetext[4]{Research Center for Statistics, University of Geneva, Boulevard du Pont d'Arve 40, 1205 Geneva, Switzerland, \texttt{sebastian.engelke@unige.ch}}

\section{Introduction}

At the focus of any risk management process lies the accurate estimation of the underlying model. Modelling assumptions, historical data, or dynamics of the driving systems are sources of errors that may affect the model accuracy. The two main sources of uncertainty come from the statistical estimation of the model parameters, and a possible misspecification of the statistical model itself. The former is known as estimation uncertainty and may be quantified by confidence intervals based on asymptotic theory or bootstrapping. The latter is called the epistemic or model uncertainty and it can be addressed by finding the worst-case model within a neighborhood around the estimated reference model $\hat P$ that accounts for the ambiguity of the model choice.

There is large body of literature on the optimal decision under model uncertainty, and in particular, the construction of ambiguity sets. \cite{Jagannathan1977} proposes ambiguity sets with given first two moments, \cite{ShapiroKleywegt2002} consider the convex hull of a finite number of models, while \cite{Calafiore2007} uses Kullback-Leibler divergence ($\KL$) around a reference model $\hat P$ to build the neighborhood. \cite{PflugWozabal2007} use the Wasserstein ball to robustify a portfolio of assets, and \cite{HansenSargent2008} derive alternative models lying at the maximal $\KL$-divergence from the reference model in a multistage case.

Risk assessment becomes crucial when dealing with rare events of a random variable $X$, as for instance $\{X >x\}$ for large $x$. Such tail events have small occurrence probabilities but may have serious impacts. Examples are financial applications where $X$ is the loss of a stock \citep{poo2004}, flood risk where $X$ is the river discharge \citep{eng2014b}, or heatwaves where $X$ is the temperature \citep{eng2017a}. A proper risk analysis is particularly important in the case where $X$ is heavy-tailed, that is, the tail probabilities decay with polynomial rate
\begin{align}\label{tail}
  P(X > x) = L(x) x^{-\beta},
\end{align}
where $L$ is a slowly varying function and $\beta > 0$ is the so-called tail index.
Extreme value theory studies asymptotically motivated approximations of these distributional tails, and numerous statistical methods for the estimation of $P(X > x)$ and $\beta$ in~\eqref{tail} exist \citep[e.g.,][]{EmbrechtsKluppelbergMikosch2013, DeHannFerreira2007}. Since these models are used to extrapolate to quantiles outside the range of the data where model assessment is notoriously difficult, model misspecification cannot be ignored. The tail index $\beta$ in~\eqref{tail} carries particular importance since it determines the heaviness of the underlying distribution and dominates estimates of extreme quantiles. 

There has been an increasing interest in the distributionally robust analysis of heavy-tailed distributions. \cite{SchneiderSchweizer2015} study bounds on the tail index in $\alpha$-divergence neighborhoods around a Pareto distribution. In the case of the Wasserstein distance, \cite{BlanchetMurthy2019} show that the worst-case tail probability can be reformulated in terms of a shift of the reference model $\hat{P}$. \cite{lam2017} adopt a non-parametric approach to determine the worst-case convex tail density that is consistent with the central part of the distribution. Their geometric assumption on the density is satisfied by most of the parametric distributions, and the resulting worst-case tail density has either a bounded support or unbounded support with probability mass at infinity.
\cite{BlanchetHeMurthy2020} study the worst-case tail distribution in a R{\'e}nyi divergence neighborhood around the reference model $\hat{P}$. They show that if the reference model belongs to the maximum domain of attraction of a generalized extreme value distribution, then the worst-case tail model also belongs to the domain of attraction of a generalized extreme value distribution, but with a modified tail index.

In order to address model uncertainty in the distributional tail, in this paper we follow a distributionally robust approach that combines stochastic optimization and extreme value theory. We study the tail probabilities and corresponding robust tail indices of the worst-case distributions in an ambiguity set around the estimated reference model $\hat P$. 
The choice of ambiguity set defines the distributions considered in the robust analysis and may have a strong influence on the worst-case tail. It is thus crucial to understand the theoretical implications of a chosen ambiguity, and the goal of this paper is to provide properties of robust tail estimates for a broad range of ambiguity sets. We assume that $\hat{P}$ is a heavy-tailed probability distribution in the sense of~\eqref{tail}. This includes a wide range of distributions such as generalized extreme value or Pareto distribution with positive shape parameters. We consider the following two classes of ambiguity sets.

\begin{enumerate}[label=(\alph*)]
\item The first ambiguity set is a neighborhood of the reference model $\hat{P}$ defined through a Wasserstein distance, a popular measure of similarity between general probability distributions (see \cite{Villani2008}). Since this distance induces fairly large ambiguity sets, we find that the worst-case tail probabilities are conservative. Moreover, the corresponding robust tail index depends only on the radius of the ambiguity set and the metric used inside the definition of the Wasserstein distance, but is independent of the tail index of the reference distribution.
\item The second class of ambiguity sets are neighborhoods of $\hat{P}$ measured in Csisz{\'a}r $f$-divergence. This class of divergences includes KL, Hellinger and total variation distance, for instance. We prove a representation of the worst-case tail probabilities and the robust tail indices in terms of the function $f$ in the definition of the divergence. As we show in several examples, for Csisz{\'a}r $f$-divergence, the worst-case tail is more sensitive to the tail of the reference model $\hat{P}$ than in the case of Wasserstein distance. Our theory yields some of the results in \cite{BlanchetHeMurthy2020} as a special case.  
\end{enumerate} 

The paper is organized as follows: in Section~\ref{Preliminaries}, we provide a short introduction of extreme value theory, fixing the notations and the basic concepts. In Section~\ref{sec:main}, we define the optimization problem in a general setting. We then solve this problem for ambiguity sets defined by Wasserstein distance and $f$-divergence and provide robust asymptotic bounds for tail probabilities. A numerical example using the Danish fire insurance claims that illustrates our pre-asymptotic and asymptotic bounds is presented in Section~\ref{Numerics}. Some additional technical proofs are gathered in the Appendix.

%===========================================================================================================

\section{Background on extreme value theory}\label{Preliminaries}

Extreme value theory provides mathematical theory and statistical tools for the analysis of rare events and the estimation of high quantiles. There are two different perspectives that give rise to different limiting distributions, which are, however, closely related. The first approach considers normalized maxima of independent and identically distributed copies $X_1,X_2,\dots,X_n$ of some random variable $X$ with distribution function $F$.  Suppose that there exist sequences $(a_n)_{n\geq 1}\in\mathbb{R}_+$ and $(b_n)_{n\geq 1}\in\mathbb{R}$ such that 
the maximum $M_n:=\max(X_1,X_2,\ldots,X_n)$ converges in distribution, that is, 
\begin{align}\label{gev}
\lim_{n\rightarrow\infty}P\bigg(\dfrac{M_n-b_n}{a_n}\leq x\bigg)= \lim_{n\rightarrow\infty} F^n(a_n x+ b_n)=H(x), \quad x\in\mathbb{R},
\end{align}
for some non-degenerate, continuous distribution function $H$. Then $F$ is said to belong to the maximum domain of attraction of $H$ and the Fisher--Tippett--Gnedenko theorem states that $H$ is a generalized extreme value distribution of the form \cite[e.g.,][]{Coles2001}
\[
  H(x)= \exp\bigg[-\bigg(1+\xi\dfrac{x-\mu}{\sigma}\bigg)_+^{-1/\xi}\bigg], \qquad x \in\mathbb R,
\]
where $\xi\in\mathbb{R}$, $\mu\in\mathbb{R}$ and $\sigma>0$ are the shape, location and scale parameters, respectively, and $a_+ = \max(0,a)$ denotes the positive part of a real number $a\in\mathbb R$. The sign of the shape parameter $\xi$ determines the type of the limiting distribution and the tail heaviness of distributions $F$ in its domain of attraction: for $\xi < 0$, $F$ has a finite upper end-point and $H$ is a Weibull distribution; for $\xi = 0$, $F$ is light-tailed and $H$ is a Gumbel distribution; for $\xi > 0$, $F$ is heavy-tailed and $H$ is a Fr{\'e}chet distribution.

The second perspective studies the distribution of the exceedances of the random variable $X$ over a high threshold $u$, that is,
\[
F_u(x):=P(X-u\leq x \mid X>u), \quad x > 0.
\]
Under the same mild assumption that guarantees the convergence in~\eqref{gev}, the distribution of exceedances converges to the generalized Pareto distribution \citep{BalkemaDeHaan1974, Pickands1975}, that is,
\begin{align}\label{gpd}
\lim_{u\uparrow x_F}\sup_{0<x<x_F-u}\vert F_u(x)-G_{\xi,\sigma(u)}(x)\vert =0, 
  \end{align}  
where $x_F=\sup\lbrace x\in\mathbb{R}:\, F(x)<1\rbrace$ is the upper end-point of $F$, and $G_{\xi, \sigma(u)}$ is the generalized Pareto distribution
\[
G_{\xi, \sigma(u)}(x)=
1-\bigg(1+\xi\dfrac{x}{\sigma(u)}\bigg)_+^{-1/\xi}, \quad x \in \mathbb R,
\]
where $\xi$ is the same shape parameter as in the generalized extreme value distribution, and $\sigma(u)>0$ is the scale parameter depending on the threshold $u$.

At this point, the importance of the tail index is worth mentioning. Not only does it characterize the heaviness of a distribution, but it is also preserved regardless of whether the generalized extreme value distribution or the generalized Pareto distribution is chosen to model extreme events. However, the estimation of the shape $\xi$ is notoriously difficult, and there is an increasing interest in developing robust tail index estimators. For instance, \cite{BrazauskasSerfling2000} focus on robust tail estimators under the assumption of an underlying Pareto model and \cite{VandewalleBeirlantChristmannHubert2007} consider a weaker domain of attraction assumption. \cite{DupuisField1998}, \cite{PengWelsh2001} and \cite{JuarezSchucany2004} propose estimators in the case where observations originate from a generalized extreme value and generalized Pareto distributions.

We focus on the case of heavy-tailed distributions, which is the most important for risk assessment.
In that case, it is common to define the tail index $\beta = 1/\xi > 0$. In order to characterize the domain of attraction in this case, we need the definition of regular variation. 

\begin{definition}\normalfont
A positive Lebesgue measurable function $L$ on $(0,\infty)$ is regularly varying at infinity of index $\gamma\in\mathbb{R}$ if
\[
\lim_{x\rightarrow\infty}\dfrac{L(tx)}{L(x)}=t^\gamma, \quad  t>0,
\] 
and we write $L\in\mathcal{R}_{\gamma}$. If $\gamma=0$, $L$ is said to be slowly varying (at infinity) and we write $L\in\mathcal{R}_0$.
\end{definition}
The random variable $X$ with distribution $F$ is in the maximum domain of attraction of the generalized extreme value distribution with positive shape $\xi = 1/\beta$ if and only if the tail $P(X > x)$ is regularly varying at infinity with index $\beta > 0$ in the sense of~\eqref{tail} \citep[e.g.,][]{EmbrechtsKluppelbergMikosch2013}. In this case, both limits in~\eqref{gev} and~\eqref{gpd} exist. Moreover, the tail index $\beta$ corresponds to the number of finite moments (see \cite{EmbrechtsKluppelbergMikosch2013}[Proposition A3.8]).

For more details on extreme value theory and regularly variation we refer to \cite{Coles2001}, \cite{Resnick2007} and \cite{DeHannFerreira2007}.

%===========================================================================================================

\section{Distributionally robust tail bounds}
\label{sec:main}
\subsection{Problem formulation}
\label{Formulation}

Let $(S,\mathcal{B}(S))$ be a measurable space, where $S$ is a Polish space, and let $\mathcal{P}(S)$ be the set of probability measures on $(S,\mathcal{B}(S))$. Assume that a reference distribution $\hat{P}\in\mathcal{P}(S)$ is obtained either from statistical analysis of historical data or expert knowledge. The resulting risk estimates depend on the chosen probability model and model uncertainty cannot be excluded. Following the stream of research in robust optimization (\cite{BenTaletal2013}, \cite{JiangGuan2016}, \cite{EsfahaniKuhn2018}), we construct a set of plausible models around the reference model $\hat{P}$:
\[
\mathcal{P}_{\mathcal D}(\delta):=\lbrace P\in\mathcal{P}(S) :\,\mathcal D(P,\hat{P})\leq\delta\rbrace,
\]
which is called the ambiguity set. The mapping $\mathcal D:\mathcal{P}(S)\times\mathcal{P}(S)\rightarrow\mathbb{R}_+$ is a measure of the discrepancy between the probability measures $P$ and $\hat{P}$ and $\delta\geq 0$ is called the tolerance level or ambiguity radius. The  measure of discrepancy is assumed to satisfy $\mathcal{D}(P,\hat{P})=0$ if and only if $P=\hat{P}$. If $\delta=0$, the ambiguity set $\mathcal{P}_{\mathcal D}(\delta)$ reduces to the singleton $\lbrace\hat{P}\rbrace$. 

In this work, we are interested in analyzing the risk associated with a univariate random variable $X$ with values in $S = \mathbb R_+ = [0,\infty)$ and distribution function $F$. As argued in Section~\ref{Preliminaries}, such a risk crucially relies on the tail of the distribution $F$. For a reference model $\hat{P}$, our aim is therefore to identify the tail behavior of the worst-case exceedance probabilities over high levels $x \in \mathbb R_+$ in a neighborhood of $\hat{P}$, that is, 
\begin{equation*}
  1 - F^*_{\mathcal D}(x) = \sup\left\{ P \lbrace (x,\infty)\rbrace :\, P\in\mathcal{P}_{\mathcal D}(\delta)\right\}, \tag{P}\label{optim}
\end{equation*}
where $\delta\geq 0$ is the tolerance level. For any $x \in \mathbb R_+$, \eqref{optim} is a separate optimization problem and it is \emph{a priori} not guaranteed that $F^*_{\mathcal D}$ is a valid distribution function. It can, however, easily be seen that right-continuity of $F^*_{\mathcal D}$ is inherited by the corresponding property of distribution functions of probability measures in $\mathcal{P}_{\mathcal D}(\delta)$. 

We are particularly interested in the tail index of the worst-case tail $ 1 - F^*_{\mathcal D}$. The choice of the discrepancy $\mathcal{D}$ will strongly influence the analysis and the goal is to study this problem for a large class of choices for $\mathcal D$.

%===========================================================================================================

\subsection{Robust tail bounds for Wasserstein distances}\label{sWasserstein}

In this section we analyze the structure of problem \eqref{optim} when the Wasserstein distance is the chosen measure of discrepancy $\mathcal{D}$ between alternative models. 

\begin{definition}\normalfont
For a metric $\mathrm{d}$ on $S$, the Wasserstein distance of order $r\geq 1$ between probability measures $P$ and $Q$ on $S$ is defined as
\[
\WD_{\mathrm{d},r}(P,Q):=\inf_{\pi\in\Pi(P,Q)}\Big[\mathbb{E}_\pi[\mathrm{d}(X,Y)]^r\Big]^{1/r},
\]
where $\Pi(P,Q):=\big\lbrace \pi\in\mathcal{P}(S\times S)\mid \pi(A\times S)=P(A), \pi(S\times A)=Q(A), \forall A\in\mathcal{B}(S)\big\rbrace$ is the set of all couplings between $P$ and $Q$. 
\end{definition}
The Wasserstein distance defined above is a metric on $\mathcal{P}(S)$; for more on the properties see \citet[Chapter 6]{Villani2008}. Problem \eqref{optim} was analysed in \cite{BlanchetMurthy2019} as a more general optimization problem
\begin{equation}
\Gamma_{\mathcal D}(\hat{P};\delta):=\sup_{P\in\mathcal{P}_{\mathcal D}(\delta)}\int_S g(y)\,dP(y), \tag{P$_1$}\label{optimG}
\end{equation}
where $g:S\rightarrow\mathbb{R}$ is an upper semi-continuous function. The strong duality of the problem above has been proven by several authors, under different assumptions \citep[e.g.,][]{GaoKleywegt2016, BlanchetHeMurthy2020, ZhaoGuan2018, EsfahaniKuhn2018}. The dual problem of \eqref{optimG} is presented in the next theorem; see Theorem 1 in \cite{BlanchetMurthy2019} for a proof.

\begin{theorem} \label{tDual}
Let $g:S\rightarrow\mathbb{R}$ be an upper semi-continuous function with $\mathbb{E}_{\hat{P}}\big[\vert g(X)\vert\big]<\infty$. For any $\lambda\geq 0$, define $\psi_\lambda:S\rightarrow\mathbb{R}\cup\lbrace\infty\rbrace$, $\psi_\lambda(x)=\sup_{y\in S}\lbrace g(y)-\lambda\mathrm{d}(x,y)\rbrace$, for $x\in S$; by convention, $\lambda \mathrm{d}(x,y)=\infty$ whenever $\lambda=0$ and $\mathrm{d}(x,y)=+\infty$. Then,
\begin{enumerate}[label=(\alph*)]
\item strong duality holds, i.e., 
\begin{equation}\label{eqDual}
\Gamma_{\WD}(\hat{P};\delta):=\sup_{\WD_{\mathrm{d},r}(P,\hat{P})\leq\delta}\mathbb{E}_P[g(X)]=\inf_{\lambda\geq 0}\Big\lbrace\lambda\delta+\mathbb{E}_{\hat{P}}\Big[\sup_{z\in S}\lbrace g(z)-\lambda\,\mathrm{d}(X,z)\rbrace\Big]\Big\rbrace;
\end{equation}
\item there exists a dual optimal solution given by the pair $(\lambda^*,\psi_{\lambda^*})$, for some $\lambda^*\geq 0$. Moreover,  $\pi^*\in\bigg\lbrace \pi\in\bigcup_{P\in\mathcal{P}(S)}\Pi(P,\hat{P}):\, \WD_{\mathrm{d},r}(P,\hat{P})\leq\delta\bigg\rbrace$ and the dual $(\lambda^*,\psi_{\lambda^*})$ are primal-dual solutions satisfying \eqref{eqDual} if and only if
\begin{gather}
g(y)-\lambda^*\mathrm{d}(x,y)  = \sup_{z\in S}\lbrace g(z)-\lambda^*\mathrm{d}(x,z)\rbrace, \quad \pi^*-a.s., \label{eqCond1}\\
\lambda^*\bigg(\int_{S\times S} \mathrm{d}(x,y)\,d\pi^*(x,y)-\delta\bigg)  =0.\label{eqCond2}
\end{gather}
Furthermore, if the primal optimal transport plan $\pi^*$ exists, then it is unique if and only if, for $\hat{P}$-almost every $x\in S$, there exists a unique $y\in S$ such that $y\in\argmax_{z\in S}\lbrace g(z)-\lambda^*\mathrm{d}(x,z)\rbrace$.
\end{enumerate}
\end{theorem}

\begin{remark}\label{rem1}\normalfont
From now on and throughout the paper we consider the special case $S=\mathbb R_+$, and we focus on the Wasserstein distance of order $r=1$ with
\begin{align}\label{rmetrik}
\mathrm{d}(y,z):=\vert\varphi(y)-\varphi(z)\vert, \qquad \varphi(y) = y^s, \quad y\in \mathbb R_+,
\end{align}
for some $s\geq 1$. We called $s$ the power of distortion. The corresponding Wasserstein distance between probability measures $P$ and $Q$ is given by 
\[
\WD_{\mathrm{d},1}(P,Q)= \int_0^{\infty} \big\vert F(\varphi^{-1}(y))-G(\varphi^{-1}(y))\big\vert\,dy = \int_0^{\infty} \big\vert F(y)-G(y)\big\vert\varphi'(y)\,dy, 
\]
where $F$ and $G$ are the distribution functions of $P$ and $Q$, respectively. The closed form expression of $\WD_{\mathrm{d},1}$ is obtained similarly to the result in \cite{Vallender1974} for $\varphi(x)=x$, and therefore it is omitted here.
\end{remark}

In the following analysis, we use the notation $\lambda^*(x)$ to emphasize the dependence of the optimal value of the Lagrange multiplier on the input argument $x\in \mathbb R_+$. 

\begin{lemma}\label{lWorst}
Consider the Wasserstein distance as in Remark~\ref{rem1} and the ambiguity set given by 
\[
\mathcal{P}_{\WD}(\delta):=\lbrace P\in\mathcal{P}(S) :\, \WD_{\mathrm{d},1}(P,\hat{P})\leq\delta\rbrace
\]
for a tolerance level $\delta\in (0,\infty)$. Then there exists some $x_0\geq 0$ large enough such that for all $x>x_0$ we have $\lambda^*(x)>0$ and the worst-case distribution in $\mathcal{P}_{\WD}(\delta)$ is of the form 
\begin{equation}\label{eqWWorst}
1-F^*_{\WD}(x):=\sup\lbrace P\lbrace (x,\infty)\rbrace:\, P \in \mathcal{P}_{\WD}(\delta) \rbrace = \hat{P}\bigg(\varphi^{-1}\bigg(\varphi(x)-\dfrac{1}{\lambda^*(x)}\bigg),\infty\bigg).
\end{equation}
\end{lemma}
\begin{proof}
We apply Theorem~\ref{tDual} with the upper semi-continuous function $g(y) = \mathbf{1}_{(x,\infty)}(y)$, $y\in \mathbb R_+$. The existence of the primal solution $\pi^*$ is not guaranteed in general. However, Proposition 2 in \cite{BlanchetMurthy2019} states that a primal solution $\pi^*$ always exists if $S$ is a compact Polish space. If the compactness condition is dropped, it is necessary to impose further topological assumptions on the cost $\mathrm{d}$ and objective function $g$. In our case, Assumptions (A3) and (A4) in \cite{BlanchetMurthy2019} are satisfied and Proposition 5 therein implies existence of the primal solution. Indeed, since for any $y,z\geq 0$, $\vert y^s-z^s\vert\geq \vert y-z\vert^s$, by the superadditive property of $y\mapsto y^s$, Assumption (A3) holds following Example (a) in \cite{BlanchetMurthy2019} after Proposition 5. Their Assumption (A4) holds trivially since our function $g$ is bounded.

Assume now that there exists a sequence $(x_n; n\geq 1)$ of positive numbers with $x_n \to \infty$ as $n\to \infty$ such that $\lambda^*_n:=\lambda^*(x_n)=0$, for all $n\geq 1$. Let $(F^*_n; n\geq 1)$ be the sequence of distribution functions corresponding to the worst-case probability measures $P_n^*$ for the Problem \eqref{optim} with $x = x_n$, and let $\hat{F}$ be the distribution function corresponding to $\hat{P}$. Note that it follows from \eqref{eqCond1} that $F^*_n(x_n)=0$ for all $n\geq 1$ since the worst-case $P_n^*$ must have all of its mass on $(x_n,\infty)$. As the Wasserstein distance between $P^*_n $ and $\hat{P}$ is bounded by $\delta$, it holds that
\begin{equation*}
\begin{split}
\delta & \geq \WD_{\mathrm{d},1}(P^*_n,\hat{P}) = \int_0^\infty\big\vert F^*_n(y)-\hat{F}(y)\big\vert\varphi'(y)\,dy \\
& = \int_0^{x_n}\big\vert F^*_n(y)-\hat{F}(y)\big\vert\varphi'(y)\,dy + \int_{x_n}^\infty\big\vert F^*_n(y)-\hat{F}(y)\big\vert\varphi'(y)\,dy \\ 
& \geq  \int_\varepsilon^{\varphi(x_n)} \hat{F}(\varphi^{-1}(y))\,dy +  \int_{x_n}^\infty \big\vert F^*_n(y)-\hat{F}(y)\big\vert\varphi'(y)\,dy \\ 
& \geq \hat{F}(\varphi^{-1}(\varepsilon))(\varphi(x_n)-\varepsilon)+ \int_{x_n}^\infty \big\vert F^*_n(y)-\hat{F}(y)\big\vert\varphi'(y)\,dy
\end{split}
\end{equation*}
for some $\varepsilon>0$ such that $\hat{F}(\varphi^{-1}(\varepsilon))>0$. Taking the limit on both sides, we obtain a contradiction since the first term tends to $\infty$. We can therefore choose some $x_0>0$ large enough such that $\lambda^*(x_0)>0$ for all $x>x_0$. 

Since $\lambda^*(x)>0$ for all $x>x_0$, the results from Section 2.4 in \cite{BlanchetMurthy2019} imply that the worst-case distribution in a Wasserstein neighbourhood is of the form: 
\[
1 - {F}^*_{\WD}(x) = \hat{P}\bigg\lbrace y:\, \inf\big\lbrace \mathrm{d}(y,z):\, z\in (x,\infty) \big\rbrace \leq\frac{1}{\lambda^*(x)}\bigg\rbrace.
\]
In particular, for $\mathrm{d}(y,z)=\vert\varphi(y)-\varphi(z)\vert$, the tail of the worst-case distribution is of the form \eqref{eqWWorst}.
\end{proof}

The next proposition considers the tail index of the worst-case distribution in a Wasserstein ball around a regularly varying reference model.
 
\begin{proposition}\label{pPower}
Let $\hat{P}\in\mathcal{R}_{-\hat{\beta}}$ be a reference model with a tail index $\hat{\beta}\geq 1$. For a distance $\mathrm{d}$ as in Remark~\ref{rem1} for some $s\geq 1$ with $s < \hat{\beta}$, consider the ambiguity set given by 
\[
\mathcal{P}_{\WD}(\delta):=\lbrace P\in\mathcal{P}(S) :\, \WD_{\mathrm{d},1}(P,\hat{P})\leq\delta\rbrace,
\]
for a tolerance level $\delta\in (0,\infty)$. The tail of the worst-case distribution in \eqref{optim} for $\mathcal{D}$ being the Wasserstein distance behaves like
\begin{equation*}
1-F_{\WD}^*(x) \sim {\delta} x^{-s},\quad x\rightarrow\infty,
\end{equation*}
that is, the worst-case tail index is $\beta^*=s.$
\end{proposition}
\begin{proof}
As $\hat{P}\in\mathcal{R}_{-\hat{\beta}}$, the corresponding distribution can be written as
\[
\hat{F}(y)= 1-y^{-\hat{\beta}}L(y),
\]
for some slowly varying function $L\in\mathcal{R}_0$. Since $s<\hat{\beta}$, then $\mathbb{E}_{\hat{P}}[ X^s]<\infty$. 

Lemma~\ref{lWorst} states that there exists some $x>0$ large enough such that $\lambda^*(x)>0$ and the worst-case distribution has support on 
\[
\big\lbrace y:\, \mathrm{d}(y,A)\leq 1/\lambda^*(x)\big\rbrace,
\]
where $A:=[x,\infty)\subset\mathbb{R}_+$ and $\mathrm{d}(y,A):=\inf\lbrace \mathrm{d}(y,z)\,: z\in A\rbrace$. Observe that $\mathrm{d}(y,A)=0$, if $y>x$ and $\mathrm{d}(y,A)=x^s-y^s$, otherwise. Hence, the support of $F_{\WD}^*$ is 
\[
\big[(x^s-1/\lambda^*(x))^{1/s}, \infty).
\]
To simplify the notation, we let $U(x):=x^s-\dfrac{1}{\lambda^*(x)}$. The worst-case distribution in $\mathcal{P}_{\WD}(\delta)$ for $\varphi(y)=y^s$ is given in terms of the reference distribution, i.e.,
\[
1-F_{\WD}^*(x)=\hat{P}\bigg(\bigg(x^s-\dfrac{1}{\lambda^*(x)}\bigg)^{1/s},\infty\bigg) = \hat{P}\big(U(x)^{1/s}, \infty\big)=U(x)^{-\hat{\beta}/s}L(U(x)^{1/s}),
\]
where $\lambda^*(x)$ solves \eqref{eqDual}. The asymptotic behavior of $F^*_{\WD}$ depends on the asymptotic behavior of $U(x)^{1/s}$, as $x\rightarrow\infty$. The slackness condition \eqref{eqCond2} is equivalent to
\begin{equation}\label{eqSlackness}
\delta = \int_{U(x)^{1/s}}^x (x^s-y^s)\,d\hat{F}(y).
\end{equation}
We claim that \eqref{eqSlackness} implies that $\lim_{x\rightarrow\infty}U(x)^{1/s}=\infty$. Assume by contradiction that there exists a sequence of positive numbers $(x_n; n\geq 1)$ with $x_n \to \infty$ as $n\to \infty$ such that $\liminf_{x_n\rightarrow\infty}U(x_n)^{1/s}=:C<\infty$. Hence there exists a sub-sequence $(x_{n_k}; n_k\geq 1)$ of positive numbers with $x_{n_k} \to \infty$ as $n_k\to \infty$ satisfying $C(1-\varepsilon)\leq U(x_{n_k})^{1/s}\leq C(1+\varepsilon)$, for some $\varepsilon>0$ and for all $n_k\geq n_0$. The slackness condition \eqref{eqSlackness} becomes
\begin{equation*}
\begin{split}
\delta & = \int_{U(x_{n_k})^{1/s}}^{x_{n_k}}(x_{n_k}^s-y^s)\,d\hat{F}(y) \\
& \geq \int_{C(1+\varepsilon)}^{x_{n_k}}(x_{n_k}^s-y^s)\,d\hat{F}(y) = x_{n_k}^s\big(\hat{F}(x_{n_k})-\hat{F}(C(1+\varepsilon))\big)-\mathbb{E}_{\hat{F}}[X^s\mathbf{1}_{\lbrace C(1+\varepsilon)\leq X\leq x_{n_k}\rbrace}]\\
& \geq x_{n_k}^s\big(\hat{F}(x_{n_k})-\hat{F}(C(1+\varepsilon))\big)-\mathbb{E}_{\hat{F}}[X^s].
\end{split}
\end{equation*}
 As $X$ has finite $s$-th moment and $x_{n_k}^s(\hat{F}(x_{n_k})-\hat{F}(C(1+\varepsilon)))\rightarrow\infty$ for $n_k\rightarrow\infty$, we get 
\[
\delta \geq \lim_{n_k\rightarrow\infty}x_{n_k}^s(\hat{F}(x_{n_k})-\hat{F}(C(1+\varepsilon)))-\mathbb{E}_{\hat{F}}[X^s] = \infty,
\]
which contradicts $\delta<\infty$.

The above analysis implies that $U(x)^{1/s}\rightarrow \infty$, as $x\rightarrow\infty$, if \eqref{eqSlackness} holds. Then we have
\begin{equation*}
\begin{split}
\delta & = \int_{U(x)^{1/s}}^x (x^s-y^s)\,d\hat{F}(y) = \int_{U(x)^{1/s}}^x x^s\,d\hat{F}(y) - \int_{U(x)^{1/s}}^x y^s\,d\hat{F}(y) \\
& = x^s\hat{F}(y)\Big\vert_{U(x)^{1/s}}^x - y^s\hat{F}(y)\Big\vert_{U(x)^{1/s}}^x + \int_{U(x)^{1/s}}^x sy^{s-1}\hat{F}(y)dy \\
& = x^s\big(1-y^{-\hat{\beta}}L(y)\big)\Big\vert_{U(x)^{1/s}}^x - y^s\big(1-y^{-\hat{\beta}}L(y)\big)\Big\vert_{U(x)^{1/s}}^x + s\int_{U(x)^{1/s}}^x y^{s-1}\big(1-y^{-\hat{\beta}}L(y)\big)dy \\
& = (x^s - U(x))U(x)^{-\hat{\beta}/s}L(U(x)^{1/s})-s\int_{U(x)^{1/s}}^\infty y^{s-\hat{\beta}-1}L(y)dy  + s\int_x^\infty y^{s-\hat{\beta}-1}L(y)dy 
\end{split}
\end{equation*}

Since $s-\hat{\beta}<0$, the Karamata theorem yields 
\[
\int_x^\infty y^{s-\hat{\beta}-1}L(y)dy\sim -\dfrac{1}{s-\hat{\beta}}x^{s-\hat{\beta}}L(x),\quad \text{ as } x\rightarrow\infty.
\]
We then obtain
\begin{align} \label{eqLimit}
\delta &\sim x^sU(x)^{-\hat{\beta}/s}L(U(x)^{1/s})-U(x)^{1-\hat{\beta}/s}L(U(x)^{1/s})+\dfrac{s}{s-\hat{\beta}}U(x)^{1-\hat{\beta}/s}L(U(x)^{1/s}) \nonumber\\
& \qquad -\dfrac{s}{s-\hat{\beta}}x^{s-\hat{\beta}}L(x) \nonumber  \\
& \sim  -\dfrac{s}{s-\hat{\beta}}x^{s-\hat{\beta}}L(x) + x^sU(x)^{-\hat{\beta}/s}L(U(x)^{1/s}) +\dfrac{\hat{\beta}}{s-\hat{\beta}}U(x)^{1-\hat{\beta}/s}L(U(x)^{1/s}).
\end{align}
As $s-\hat{\beta}<0$ and $\lim_{x\rightarrow\infty}U(x)=\infty$, the first and last terms in \eqref{eqLimit} converge to $0$, as $x\rightarrow\infty$. It follows that the dominant term is the middle one and, as $\lim_{x\rightarrow\infty}x^s=\infty$ for $s\geq 1$, it implies that 
\begin{equation*}
1-F_{\WD}^*(x)= U(x)^{-\hat{\beta}/s}L(U(x)^{1/s})\sim {\delta} x^{-s},\quad x\rightarrow\infty.
\end{equation*}
Hence the worst-case tail distribution decays as $x^{-s}$, while the ambiguity radius $\delta$ is absorbed as a scaling parameter. 
\end{proof}

Proposition~\ref{pPower} shows that the tail index of the worst-case distributions in a Wasserstein ball does not depend on the tail index of the reference distribution. In the extreme case, when $s=1$, the worst-case model in $\mathcal{P}_{\WD}(\delta)$ has a very heavy tail, independent of the reference tail index $\hat{\beta}$. The reason for this surprising behavior is that one can find two probability measures lying in a Wasserstein ball of any small radius $\delta$, but whose tails have very different decay. 

%===========================================================================================================

\subsection{Robust tail bounds for $f$-divergence}\label{sDivergence}

Another common choice for discrepancy $\mathcal{D}$ in a robust optimization context is the class of $f$-divergences, including for instance the $\KL$-divergence and the Hellinger divergence. There is significant amount of literature that uses these discrepancies to quantify the implication of model error in financial and actuarial risk measurements \citep[e.g.,][]{DupuisJamesPetersen2000, HansenSargent2001, BenTaletal2013,GlassermanXu2014}. In particular, \cite{BreuerCsiszar2016} and \cite{CsiszarBreur2018} evaluated the maximum expected loss of a portfolio in an ambiguity set $\mathcal{P}$ constructed with respect to $f$- and Bregman divergence, and characterized the density of the worst-case distribution in $\mathcal{P}$. To the best of our knowledge, the impact of distributional uncertainty in the context of extreme value theory has been investigated only by \cite{BlanchetHeMurthy2020}, where the R{\'e}nyi divergence is chosen as a measure of discrepancy between distributions. 

In this section we study the same problem \eqref{optim} when the measure of discrepancy is the $f$-divergence. 

\begin{definition}\normalfont\label{def:fdiv}
Let $f:(0,\infty)\rightarrow \mathbb{R}$ be a convex function such that $f(1)=0$. For $P,Q \in \mathcal P(S)$, suppose that $P$ is absolutely continuous with respect to $Q$. The $f$-divergence between $P$ and $Q$ is 
\[
D_f(P\Vert Q)=\mathbb{E}_{Q}\bigg[f\bigg(\dfrac{dP}{dQ}\bigg)\bigg].
\]
\end{definition}
By convention, $f(0):=\lim_{y\rightarrow 0^+}f(y)\in(-\infty,\infty]$, which exists because of the convexity of $f$ \citep[see][Lemma 2.1]{LieseVajda2008}. The $*$-conjugate of $f$ is defined as 
\[
f^*(y)=yf\Big(\dfrac{1}{y}\Big), \, y\in(0,\infty),
\]
which is also a convex function and we define $f^*(0):=\lim_{y\rightarrow 0^+}f^*(y)\in(-\infty,\infty]$.
The $f$-divergence enjoys the following properties (see \cite{Csiszar1967}, \cite{Csiszar1974}):
\begin{enumerate}
\item Non-negativity: $D_f(P\Vert Q)\geq 0$ with equality if and only if $f$ is strictly convex at $1$ and $P=Q$.
\item Symmetry: $D_f(P\Vert Q)=D_{f^*}(Q\Vert P)$.
\item Joint convexity: for $\lambda\in [0,1]$ it holds
\[
D_f(\lambda P_1+ (1-\lambda)P_2\Vert \lambda Q_1 + (1-\lambda)Q_2)\leq \lambda D_f(P_1\Vert Q_1) + (1-\lambda)D_f(P_2\Vert Q_2).
\]
\item Monotonicity: $D_f(P_{\vert \mathcal{G}}\Vert Q_{\vert \mathcal{G}}) \leq D_f(P\Vert Q)$, for any sub-$\sigma$-algebra $\mathcal{G}\subseteq \mathcal{B}(S)$.
\item Range of values: $D_f(P\Vert Q)\leq f(0)+f^*(0)$.
\item Lower semi-continuous in the pair $(P, Q)$ in the weak topology: for $(P_n)_{n\geq 1}$ and $(Q_n)_{n\geq 1}$ sequences of distributions that weakly converge to $P$ and $Q$, respectively, it holds that
\[
\liminf_{n\rightarrow\infty} D_f(P_n\Vert Q_n)\geq D_f(P\Vert Q).
\]
\item For an arbitrary $b\in\mathbb{R}$, if the divergence $\tilde{f}$ is defined as $\tilde{f}(y):=f(y)+ b(y-1)$, for all $y\geq 0$, then $D_{\tilde{f}}(P\Vert Q)= D_f(P\Vert Q)$, for any pair $(P,Q)$, with $P\ll Q$.
\end{enumerate}

\begin{example}\label{exF}\normalfont For particular choices of the function $f$ one can recover known divergences:
\begin{enumerate}[(i)]
\item Kullback--Leibler divergence: $\KL (P\Vert Q)=\mathbb{E}_Q\bigg[\dfrac{dP}{dQ}\log\bigg(\dfrac{dP}{dQ}\bigg)\bigg]$ for $f(y)=y\log(y)$.
\item Jeffrey's divergence: $\text{Jeffrey}(P\Vert Q) = \mathbb{E}_Q\bigg[\bigg(\dfrac{dP}{dQ}-1\bigg)\log\bigg(\dfrac{dP}{dQ}\bigg)\bigg]$ for $f(y)=(y-1)\log(y)$.  
\item Hellinger divergence of order $\alpha\in (1,\infty)$: $\mathcal{H}_\alpha(P\Vert Q)=D_f(P\Vert Q)$ for $f(y)=\dfrac{y^\alpha-1}{\alpha-1}$. 
\item $\chi^2$-divergence is the Hellinger divergence of order $2$, i.e., $\chi^2(P\Vert Q)=\mathbb{E}_Q\bigg[\dfrac{dP}{dQ}-1\bigg]^2$ for $f(y)=y^2-1$ or  $f(y)=(y-1)^2$.
\item Total variation distance: $\vert P-Q\vert = D_f(P\Vert Q)$ for $f(y)=\vert y-1\vert$.
\item Triangle discrimination: $\Delta(P\Vert Q)=D_f(P\Vert Q)$ for $f(y)=(y-1)^2/(y+1)$.
\item Jensen--Shannon divergence: $\JS (P\Vert Q)=D_f(P\Vert Q)$ for $f(y)=y\log(y)-(1+y)\log\Big(\dfrac{1+y}{2}\Big)$.
\end{enumerate}
\end{example}

We denote by $\mathcal{C}^1(0,\infty)$ the class of continuously differentiable functions on $(0,\infty)$. For the rest of this analysis, we impose the following assumption.

\begin{assumption}\label{aF1}
  Let $f\in \mathcal{C}^1(0,\infty)$ be strictly convex such that $f'(1)=0$ and $f\vert_{[1,\infty)}$ is a positive, increasing and regularly varying function of index $\rho\geq 1$. 
\end{assumption}

 The next remark shows that the above assumption is not very restrictive.

 \begin{remark}\label{rAssump}\normalfont
  Assumption~\ref{aF1} can be relaxed since different divergence functions may lead to the same $f$-divergence between $P$ and $Q$. Indeed, if $f\in \mathcal{C}^1(0,\infty)$ is a strictly convex divergence function that is regularly varying with index $\rho\geq 0$, then we may replace $f$ with $\tilde{f}:(0,\infty)\rightarrow \mathbb{R}$, $\tilde{f}(y):=f(y)-f'(1)(y-1)$. The function $\tilde f$ then satisfies Assumption~\ref{aF1} and induces the same $f$-divergence as $f$. For a proof see Appendix~\ref{appendixA}.
\end{remark}

Figure~\ref{fig:FandFtilde} shows the divergence functions $f$ from Examples~\ref{exF} (left-hand side) together with the transformation $\tilde{f}$ from Remark~\ref{rAssump} (right-hand side).

\begin{figure}[H]
  \begin{subfigure}{\linewidth}
  \includegraphics[width=.5\linewidth]{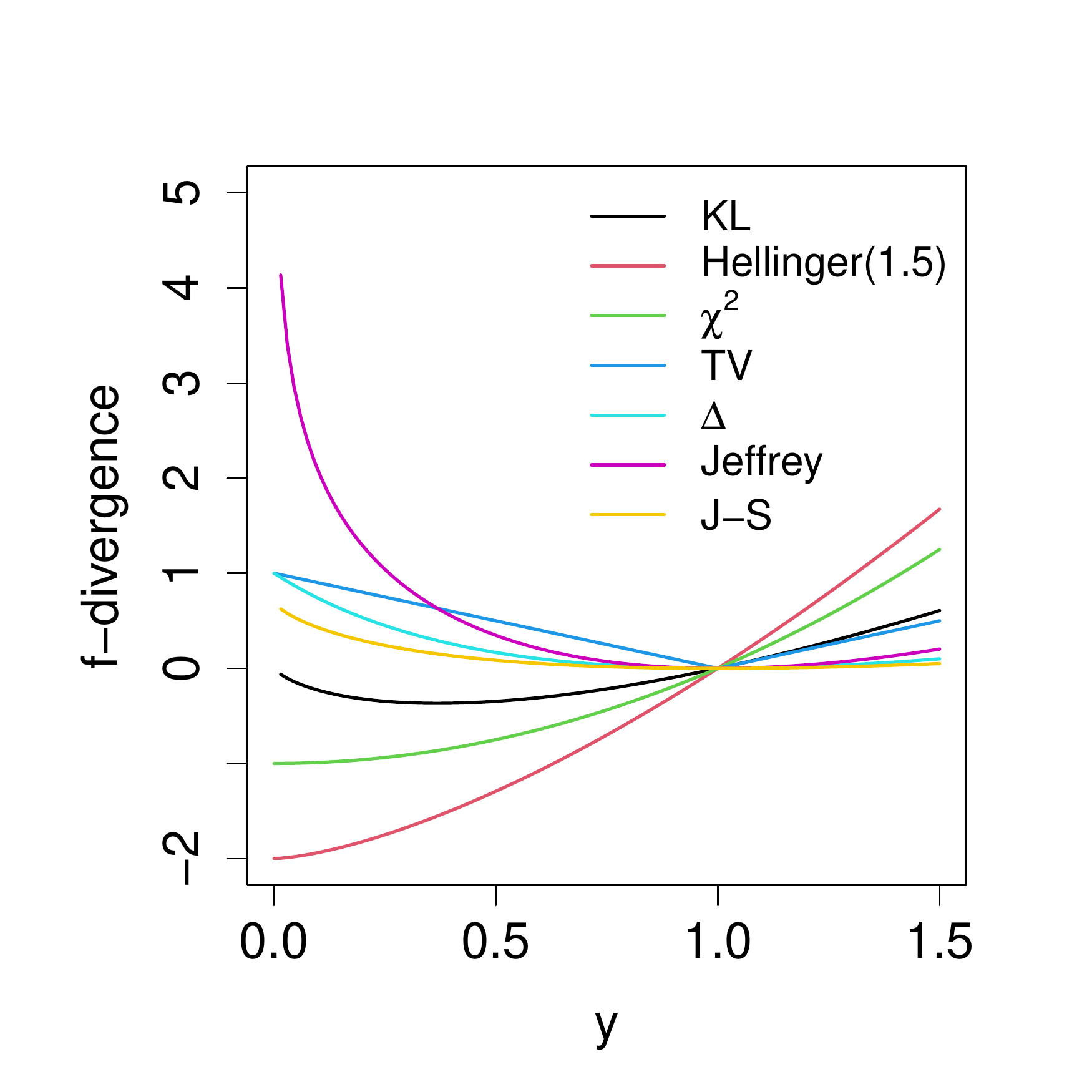}\hfill
  \includegraphics[width=.5\linewidth]{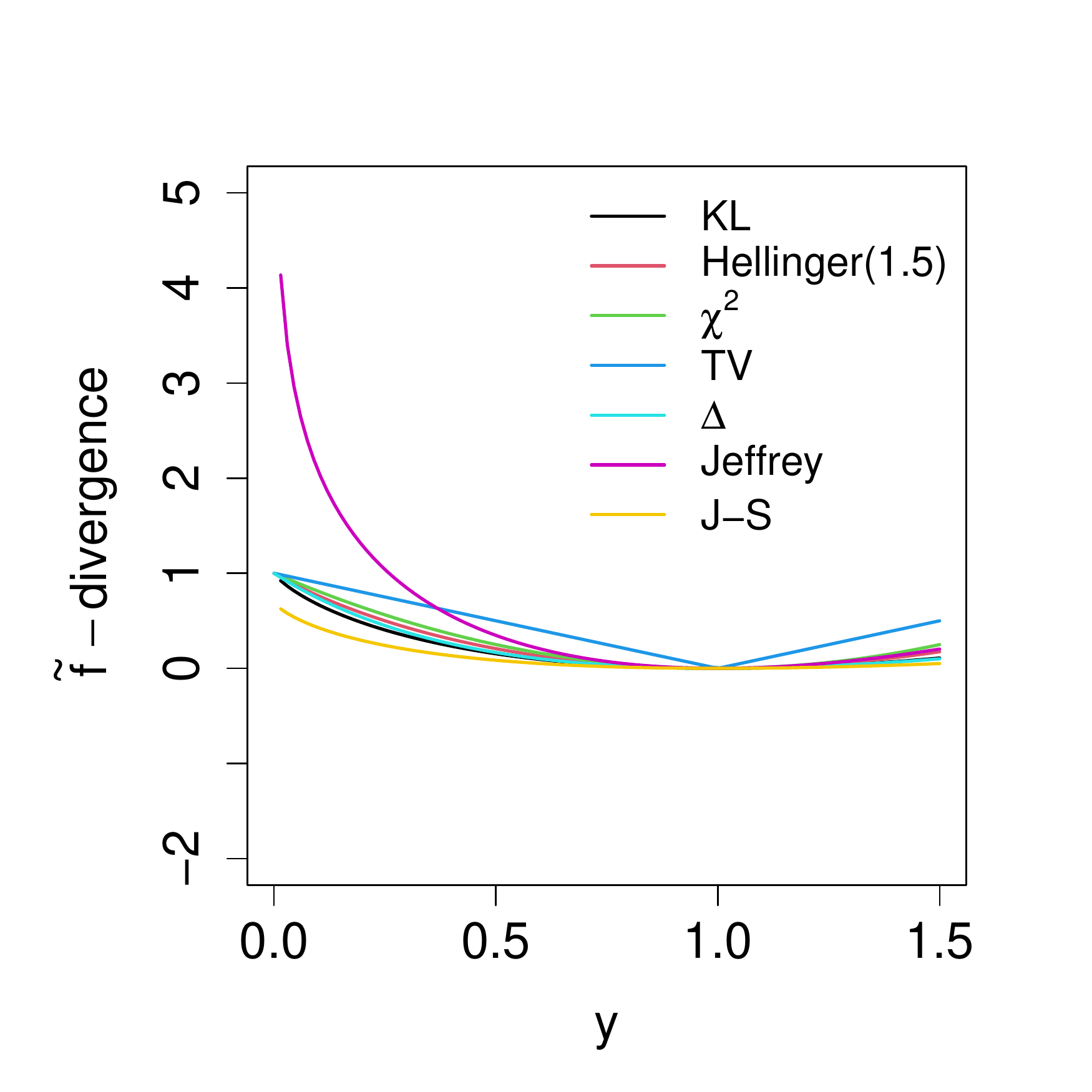}\hfill
  \end{subfigure}\vspace*{0.5cm}
  \caption{Left: divergence function $f(y)$. Right: the divergence function $\tilde{f}(y) = f(y)- f'(1)(y-1)$.}
\end{figure}\label{fig:FandFtilde}

Assumption~\ref{aF1} is satisfied by all $f$-divergences listed in Example \ref{exF}, except for the total variation distance. For example, the function $f$ corresponding to the Hellinger divergence of order $\alpha>1$ is strictly increasing and positive on $[1,\infty)$, and it is straightforward to see that $f\in\mathcal{R}_{\alpha}$.

%-----------------------------------------------------------------------------------------------------------------------------------------------------

In light of Definition~\ref{def:fdiv}, consider the optimization problem~\eqref{optim} when the neighborhood around $\hat{P}$ is the $f$-divergence ball:
\begin{equation}\label{eqWorst0}
\overline{F}^*_f(x):=\sup_{P \ll \hat P}\lbrace P\lbrace (x,\infty)\rbrace :\, D_f(P\Vert \hat{P})\leq\delta\rbrace.
\end{equation}
Observe that if $\delta> f(0)+f^*(0)\geq D_f(P\Vert \hat{P})$, then the $f$-divergence ambiguity set 
\begin{equation}\label{eqPfDelta}
\mathcal{P}_f(\delta):=\lbrace P\in\mathcal{P}(S) :\, D_f(P\Vert \hat{P})\leq \delta\rbrace
\end{equation}
becomes so large that it contains all distributions that are absolutely continuous with respect to $\hat P$. Hence, it is natural to consider a more restrictive neighborhood around $\hat{P}$ and impose the condition that $0<\delta<f(0)+f^*(0)$.

The lemma below shows that if $S$ is a compact Polish space, then the ambiguity set $\mathcal{P}_f(\delta)$ is compact with respect to the weak topology. 

\begin{lemma}\label{lFexistence}
Let $\hat{P}$ be a reference model and let $0<\delta<f(0)+f^*(0)$. If $S$ is a compact Polish space, then there exists some $P^*\in \mathcal{P}_f(\delta)$ that is optimal for Problem~\eqref{eqWorst0}. 
\end{lemma}
\begin{proof}
For $\delta\in (0, f(0)+f^*(0))$, we claim that the ambiguity set $\mathcal{P}_f(\delta)$ in \eqref{eqPfDelta} is compact in the topology of weak convergence. To see this, let $(P_n)_{n\geq 1}$ be a sequence of distributions in $\mathcal{P}_f(\delta)$ that converges weakly to a probability distribution $P$ on $(S,\mathcal{B}(S))$. As $D_f(\cdot\Vert \hat{P})$ is lower semi-continuous, it follows that $\mathcal{P}_f(\delta)$ is a closed set \cite[Proposition 1.7]{Cioranescu2012} and hence $P\in \mathcal{P}_f(\delta)$. Since $\mathcal{P}(S)$ is weak$^*$ compact and $\mathcal{P}_f(\delta)$ is a closed subset in $\mathcal{P}(S)$, it follows that $\mathcal{P}_f(\delta)$ is compact. Since the mapping $P \mapsto P(x,\infty)$ is upper semi-continuous, it follows that there exists an optimizer $P^* \in \mathcal{P}_f(\delta)$ for Problem~\eqref{eqWorst0} \citep[see][Theorem B.2]{Puterman2014}.
\end{proof}

\begin{remark}\normalfont
In the absence of compactness of $S$, the compactness of $\mathcal{P}_f(\delta)$ can still be proven for particular choices of divergence (see \cite{VanErvenHarremos2014} for the case of R{\'e}nyi or Kullback--Leibler divergences). 
\end{remark} 

Throughout the rest of the section, we consider the special case of the compactified positive real line $S=[0,\infty]$. According to Lemma~\ref{lFexistence}, we thus guarantee that the supremum in Problem~\eqref{eqWorst0} is attained.

To obtain a solution of the Problem~\eqref{eqWorst0}, let us denote the Radon--Nikodym derivative by $L:= dP/d\hat{P}$ and \eqref{eqWorst0} becomes
\begin{equation}\label{eqWorst}
\overline{F}^*_f(x):=\sup_{L\geq 0}\lbrace\mathbb{E}_{\hat{P}}[L\mathbf{1}_{\lbrace X>x\rbrace}]: \mathbb{E}_{\hat{P}}[f(L)]\leq\delta, \mathbb{E}_{\hat{P}}[L]=1\rbrace.
\end{equation}
The next result presents the solution of the optimization problem \eqref{eqWorst}. The proof follows a similar strategy as in \cite{HuHong2013} that treat the Kullback--Leibler divergence case, or \cite{EngelkeIvanovs2017} for the $L^2$-distance.

\begin{proposition}\label{POptim}
  Let $\hat{P}$ be a reference model and let $x \geq 0$ such that $\hat{P}(x,\infty)>0$. Let $f:(0,\infty)\rightarrow\mathbb{R}$ be a divergence satisfying Assumption~\ref{aF1}. Then, for $\delta\in (0,f(0)+f^*(0))$, $L^*\geq 0$ is the optimizer of Problem~\eqref{eqWorst} if and only if $\mathbb{E}_{\hat{P}}[L^*]=1$ and at least one of the following cases holds:
  \begin{thmcases}
  \item \label{itm:firstCase} there exists some $\lambda_1>0$ and $\lambda_2\in\mathbb{R}$ such that $L^*$ solves $\mathbf{1}_{\lbrace{X>x}\rbrace}+\lambda_2-\lambda_1f'(L^*)=0$, $\hat{P}$-a.s.~and $\mathbb{E}_{\hat{P}}[f(L^*)]=\delta$. 
  \item \label{itm:secondCase} $\hat P(L^*=0, X \leq x) = 1$ and $\mathbb{E}_{\hat{P}}[f(L^*)]\leq \delta$. 
  \end{thmcases}
  \end{proposition}
  \begin{proof}
  The proof is based on the dual formulation of the Problem \eqref{eqWorst}. It is easy to see that \eqref{eqWorst} is a convex optimization problem with the corresponding Lagrange function defined as
  \[
  \mathcal{L}(L,\lambda_1,\lambda_2)= \mathbb{E}_{\hat{P}}[\mathbf{1}_{\lbrace{X>x}\rbrace}L] -\lambda_1(\mathbb{E}_{\hat{P}}[f(L)]-\delta) + \lambda_2(\mathbb{E}_{\hat{P}}[L]-1),
  \]
  for some $\lambda_1\geq 0$ and $\lambda_2\in\mathbb{R}$. Observe that for $\tilde{L}:=1$ $\hat{P}$-a.s., $\mathbb{E}_{\hat{P}}[f(\tilde{L})]=0<\delta$ and $\mathbb{E}_{\hat{P}}[\tilde{L}]=1$; the Slater's condition implies the strong duality \citep[see][Section 5.2.3]{BoydVandenberghe2004} and thus $L^*$ is the optimal solution of \eqref{eqWorst} if and only if there exist some $\lambda_1\geq 0$ and $\lambda_2\in\mathbb{R}$ and $L^*$ is the optimal solution of 
  \begin{equation}
  \sup_{L\geq 0}\mathcal{J}(L):= \sup_{L\geq 0}\mathbb{E}_{\hat{P}}[(\mathbf{1}_{\lbrace{X>x}\rbrace}+\lambda_2)L-\lambda_1 f(L)],\label{eqInner}
  \end{equation}
  and in addition to the constraints from Problem~\eqref{eqWorst}, the complementary slackness conditions 
  \[
  \lambda_2(\mathbb{E}_{\hat{P}}[L^*]-1)=0\text{  and }\lambda_1(\mathbb{E}_{\hat{P}}[f(L^*)]-\delta)=0
  \]
  must hold. Problem \eqref{eqInner} is concave in $L$, since $f$ is a convex function. The directional derivative of $\mathcal{J}(L)$ at $L$ in the direction of $V:[0,\infty]\rightarrow [0,\infty)$ is:
  \begin{equation*}
    \nabla_V\mathcal{J}(L) = \lim_{t\rightarrow 0}\dfrac{\mathcal{J}(L+ tV)-\mathcal{J}(L)}{t} = \mathbb{E}_{\hat{P}}[(\mathbf{1}_{\lbrace{X>x}\rbrace}+\lambda_2)V]-\lambda_1\lim_{t\rightarrow 0}\dfrac{\mathbb{E}_{\hat{P}}[f(L+tV)-f(L)]}{t}.
  \end{equation*}
  Note that the function $f(y)$ is convex on $(0,\infty)$ and hence, for any $y$ and direction $v\geq 0$, the function $[f(y+tv)-f(y)]/t$ is monotonic in $t$. To see this, let $0<t_1<t_2$; for any direction $v$, we have: 
  \[
  f(y+t_1v) = f\bigg[\dfrac{t_1}{t_2}(y+t_2v)+\bigg(1-\dfrac{t_1}{t_2}\bigg)y\bigg]\leq \dfrac{t_1}{t_2}f(y+t_2v)+\bigg(1-\dfrac{t_1}{t_2}\bigg)f(y), 
  \]
  where the last inequality uses the convexity of $f$. Multiplying both sides by $1/t_1>0$ and rearranging the terms yields $[f(y+t_1v)-f(y)]/t_1 \leq [f(y+tv_2)-f(y)]/t_2$ and hence $[f(y+tv)-f(y)]/t$ is monotonic in $t$. The monotone convergence theorem yields
  \begin{equation*}
  \begin{split}
  \nabla_V\mathcal{J}(L) & = \mathbb{E}_{\hat{P}}[(\mathbf{1}_{\lbrace{X>x}\rbrace}+\lambda_2)V]-\lambda_1\mathbb{E}_{\hat{P}}\bigg[\lim_{t\rightarrow 0}\dfrac{f(L+tV)-f(L)}{t}\bigg] \\
  & = \mathbb{E}_{\hat{P}}[(\mathbf{1}_{\lbrace{X>x}\rbrace}+\lambda_2-\lambda_1 f'(L))V].
  \end{split}
  \end{equation*}
  The optimal $L^*$ satisfies $\nabla_V\mathcal{J}(L^*) =0$, for all directions $V\geq 0$, and consequently, $L^*$ is optimal if and only if it solves the equation $\mathbf{1}_{\lbrace{X>x}\rbrace}+\lambda_2-\lambda_1f'(L^*)=0$ $\hat P$-a.s. In the case $\lambda_1>0$ the complementary slackness conditions state that $\mathbb{E}_{\hat{P}}[f(L^*)]=\delta$.

  If $\lambda_1=0$, then Problem~\eqref{eqInner} becomes linear in $L$. Again, for all $V\geq 0 $, it holds $\mathbb{E}_{\hat{P}}[(\mathbf{1}_{\lbrace{X>x\rbrace}}+\lambda_2)V]=0$, i.e., $\hat{P}(L^* > 0, \mathbf{1}_{\lbrace{X>x\rbrace}}+\lambda_2 =  0) = 1$. Since the indicator function takes only two values, it implies that $\lambda_2\in\lbrace{0, -1\rbrace}$. 
  
  If $\lambda_2=0$ then we must have $\mathbf{1}_{\lbrace{X>x}\rbrace}=0$ $\hat P$-a.s., which contradicts the assumption.
  
  If $\lambda_2=-1$ then $\hat{P}(L^* > 0, X > x)=1$, and the corresponding optimal value is $\mathbb{E}_{\hat{P}}[L^*\mathbf{1}_{{\lbrace X>x\rbrace}}]=\mathbb{E}_{\hat{P}}[L^*]=1$.
\end{proof}
  
The next proposition provides a characterization of the decay of $1-F^*_f(x)$, in terms of the reference distribution $1-\hat{F}(x)$ and the tolerance level $\delta$.

\begin{proposition}\label{Pfdivergence}
  Let $\hat{P}$ be a reference model and let $x_0\geq 0$ such that $p_x:=\hat{P}(x,\infty)>0$ for all $x \geq x_0$ and $\lim_{x\to \infty} p_x = 0$. Let $f:(0,\infty)\rightarrow\mathbb{R}$ be a divergence satisfying Assumption~\ref{aF1}. Consider the ambiguity set given by 
  \[
  \mathcal{P}_f(\delta):=\lbrace P\in\mathcal{P}(S) :\, D_f(P\Vert \hat{P})\leq \delta\rbrace,
  \]
  for a tolerance level $\delta < f^*(0) + f(0)$. The worst-case tail distribution in \eqref{optim} for $\mathcal{D}$ being the $f$-divergence is then of the form
  \begin{equation*}
    \left\{
      \begin{aligned}
        & \overline{F}^*_f(x) \sim  f^\leftarrow(\delta/p_x)p_x,\quad  \mbox{if } f^*(0) = \infty,\\
        & \overline{F}^*_f(x) \sim \ell \in (0,1], \quad \mbox{if }  f^*(0) \in (0, \infty),\\
      \end{aligned}
    \right.
  \end{equation*} 
  where in the second case $\ell$ is the unique root of the equation $\ell f^*(0) - f(1-\ell)=\delta$, and $f^\leftarrow (y):=\inf \lbrace z\geq 1:\, f(z)\geq y\rbrace$.
\end{proposition}
\begin{proof}
  According to Problem~\eqref{eqWorst}, the worst-case distribution is given by $\overline{F}^*_f(x)=\mathbb{E}_{\hat{P}}[L^*\mathbf{1}_{\lbrace X>x\rbrace}]$, where the optimal $L^*$ is one of the form in Proposition~\ref{POptim}. 
  
  In \ref{itm:firstCase} of Proposition~\ref{POptim}, to find a closed-form solution $L^*$, it is sufficient to determine constants $\lambda_1> 0$ and $\lambda_2\in\mathbb{R}$ such that $(\mathbf{1}_{\lbrace{X>x}\rbrace}+\lambda_2)/\lambda_1\in\text{im}(f')$, $\hat{P}$-a.s., i.e.,
  \begin{equation*}
    \begin{split}
    L^*  & =(f')^{-1}\bigg(\dfrac{\mathbf{1}_{\lbrace X>x\rbrace }+\lambda_2}{\lambda_1}\bigg)= (f')^{-1}\bigg(\dfrac{\lambda_2}{\lambda_1}\bigg)\mathbf{1}_{\lbrace X<x\rbrace}+(f')^{-1}\bigg(\dfrac{\lambda_2+1}{\lambda_1}\bigg)\mathbf{1}_{\lbrace X>x\rbrace} \\
    & = a_x\mathbf{1}_{\lbrace X<x\rbrace }+ b_x\mathbf{1}_{\lbrace X>x\rbrace},
    \end{split}
  \end{equation*}
  for $a_x:=\lambda_2\lambda_1^{-1}\in\mathbb{R}$ and $b_x := (\lambda_2+1)\lambda_1^{-1}\in\mathbb{R}_+$. As $\mathbb{E}_{\hat{P}}[f(L^*)]=\delta$, it holds that 
  \[ 
  \delta =f(a_x)\hat{P}(X\leq x)+ f(b_x)\hat{P}(X>x).
  \]
  Since $\mathbb{E}_{\hat{P}}[L^*]=1$, we also have 
  \begin{equation}\label{eqAxBx}
  a_x\hat{P}(X<x)+b_x\hat{P}(X>x)=1. 
  \end{equation}
  Let $p_x:=\hat{P}(x,\infty)$. If there exists some $b_x\in (1,1/p_x)$ such that 
  \begin{equation}\label{Fequation}
  (1-p_x)f\bigg(\dfrac{1-p_xb_x}{1-p_x}\bigg) + p_xf(b_x)=\delta,
  \end{equation}
  then the optimal value of \eqref{eqWorst} is of the form $\overline{F}^*_f(x)=\mathbb{E}_{\hat{P}}[L^*\mathbf{1}_{\lbrace X>x\rbrace}]=\mathbb{E}_{\hat{P}}[b_x\mathbf{1}_{\lbrace X>x\rbrace}]=b_xp_x$. To see this, define the function  $T:[1,1/p_x]\rightarrow\mathbb{R}$, $T(y):=p_xf(y)+(1-p_x)f\Big(\dfrac{1-yp_x}{1-p_x}\Big)-\delta$, continuous and increasing in $y$. Since $f(1)=0$ then $T(1)=p_xf(1)+(1-p_x)f(1)-\delta=-\delta<0$. 

  Since $f^*(p_x)+(1-p_x)f(0) \to f^*(0) + f(0) < \delta$ as $x\to\infty$, we can choose $x_0$ large enough such that for all $x \geq x_0$ we have $f^*(p_x)+(1-p_x)f(0) < \delta$. Then $T(1/p_x)>0$ and hence there exists some $b_x\in (1,1/p_x)$ such that $T(b_x)=0$. By definition, we have $b_x=(f')^{-1}\big((\lambda_2+1)/\lambda_1\big)$, and from \eqref{eqAxBx}, we also get $a_x=(1-p_xb_x)(1-p_x)^{-1}$. To determine $\lambda_1>0$ and $\lambda_2\in\mathbb{R}$, it suffices to solve the linear system of equations:
  \begin{equation*}
    \left\{
      \begin{aligned}
        & \lambda_2 - f'(b_x)\lambda_1  = -1\\
        & \lambda_2 - f'(a_x)\lambda_1  =  0.
      \end{aligned}
    \right.
  \end{equation*} 
  By solving the system, we obtain $\lambda_1=(f'(b_x)-f'(a_x))^{-1}$ and $\lambda_2=(f'(b_x)-f'(a_x))^{-1}f'(a_x)$. Observe that $b_x>a_x$ and since $f'$ is strictly increasing, then $\lambda_1>0$. 

  If $\lim_{y\rightarrow\infty}f(y)/y= f^*(0)=\infty$, then according to Lemma~\ref{lPxbx} in Appendix~\ref{appendixB}, $p_xb_x\rightarrow 0$, $x\rightarrow\infty$. We thus obtain
  \[
    \lim_{x\rightarrow\infty}(1-p_x)f\Big(\dfrac{1-p_xb_x}{1-p_x}\Big)=f(1) = 0.
  \]

  From \eqref{Fequation}, one gets $p_xf(b_x)=\delta-\varepsilon(x)$ for some $\varepsilon(x)\rightarrow 0$ as $x\rightarrow\infty$. As $f:[1,\infty)\rightarrow [0,\infty)$ is regularly varying with index $\rho\geq 1$, then there exists $f^\leftarrow :[0,\infty)\rightarrow [1,\infty)$, $f^\leftarrow (y)=\inf \lbrace y\geq 1:\, f(z)\geq y\rbrace$ such that $f^\leftarrow\in\mathcal{R}_{1/\rho}$ and 
  \[
    f(f^\leftarrow (y))\sim f^\leftarrow (f(y))\sim y,  \quad \text{as } y\rightarrow\infty.
  \]
  Hence the worst-case tail distribution in a $f$-divergence ball is given by 
  \[  
    \overline{F}^*_f(x)=p_xb_x= f^\leftarrow\big(\{\delta-\varepsilon(x)\}p_x^{-1}\big)p_x.
  \] 
  As $p_x$ and $f^\leftarrow$ are right-continuous functions, then $F^*_f$ is right-continuous as well. Further,  $f^\leftarrow$ satisfies 
  \[
    \lim_{\substack{t\rightarrow\infty \\ \lambda\rightarrow 1}}f^\leftarrow (t\lambda)/f^\leftarrow (t)=1,
  \]
  and according to Lemma 2 in \cite{Djurvcic1998}, $f^\leftarrow$ preserves the asymptotic equivalence relation, i.e., as 
  \begin{equation}\label{eqAsymptotic}
  \{\delta-\varepsilon(x)\}p_x^{-1}\sim \delta p_x^{-1}\Rightarrow f^\leftarrow\big(\{\delta-\varepsilon(x)\}p_x^{-1} \big)\sim f^\leftarrow(\delta p_x^{-1}),\quad x\rightarrow\infty.
  \end{equation}
  Then the worst-case tail $\overline{F}^*_f(x)$ is asymptotically equivalent to $f^\leftarrow\big(\delta p_x^{-1}\big)p_x$.

  If $\lim_{y\rightarrow\infty}f(y)/y= f^*(0) \in (0,\infty)$, then from Appendix~\ref{appendixB}, $\overline{F}^*_f(x)= p_xb_x\sim \ell$, where $\ell\in(0,1]$ uniquely solves the equation $\ell f^*(0) -f(1-\ell)=\delta$.
\end{proof}

For specific choices of $f$-divergence, Proposition~\ref{Pfdivergence} describes the tail behavior of the worst-case distribution in the corresponding ambiguity set. 

\begin{example}\label{exDivergence}\normalfont
  Let the reference model $\hat{P}$ be a generalized extreme value distribution with tail index $\hat{\beta}>0$, given by 
  \[
  \hat{F}(x)=\exp\bigg\lbrace-\bigg(1+\dfrac{x-\hat \mu}{\hat{\beta}\hat{\sigma}}\bigg)^{-\hat{\beta}}\bigg\rbrace,
  \]
  with scale $\hat{\sigma}>0$ and location $\hat \mu>0$. The tail of the reference model $p_x=1-\hat{F}(x)$ behaves like
  \[ 
    p_x\approx \bigg(\dfrac{x}{\hat{\beta}\hat{\sigma}}\bigg)^{-\hat{\beta}}, \quad  x\rightarrow\infty.
  \]
  Proposition~\ref{Pfdivergence} provides the asymptotic tails of the worst-case distribution in a $f$-divergence neighborhood of radius $\delta < f^*(0)+ f(0)$. Table~\ref{table:examples} illustrates the tail index and the scale parameter of the worst-case distribution for different divergence functions $f$ that satisfy our assumptions. The detailed computations are provided in Appendix~\ref{appendixC}.
  
  \begin{table}[H]
  \centering 
  \begin{tabular}{|c|c|}
  \hline
  $f$-divergence & $\overline{F}^*_f(x) $ \\ \hhline{|=|=|}
  \begin{tabular}[c]{@{}c@{}}$\KL$\\ $f(y)=y\log(y)$\end{tabular} & $\delta\hat{\beta}^{-1}\log^{-1}(x)$ \\ \hline
  \begin{tabular}[c]{@{}c@{}}Hellinger\\ $f(y)=\dfrac{y^\alpha-1}{\alpha-1}$\end{tabular} & \begin{tabular}[c]{@{}c@{}}$\Big(\dfrac{x}{\beta^*\sigma^*}\Big)^{-\beta^*}$ \\ where $\beta^*=\dfrac{\alpha-1}{\alpha}\hat{\beta}$, $\sigma^*=\alpha(\alpha-1)^{\frac{1}{\hat{\beta}(\alpha-1)}-1}\delta^{\frac{1}{\hat{\beta}(\alpha-1)}}\hat{\sigma}$\end{tabular} \\ \hline
  \begin{tabular}[c]{@{}c@{}}$\chi^2$ \\ $f(y)=y^2-1$\end{tabular} & \begin{tabular}[c]{@{}c@{}}$\Big(\dfrac{x}{\beta^*\sigma^*}\Big)^{-\beta^*}$ \\ where $\beta^*=\hat{\beta}/2$, $\sigma^*=2\delta^{1/\hat{\beta}}\hat{\sigma}$ \end{tabular} \\ \hline
  \begin{tabular}[c]{@{}c@{}}$\Delta$ \\  $f(y)=\frac{(y - 1)^2}{y + 1}$\end{tabular} & \begin{tabular}[c]{@{}c@{}}$\dfrac{2\delta}{\delta + 2}$, \, $\delta\in (0,2)$ \\ \end{tabular} \\ \hline
  \begin{tabular}[c]{@{}c@{}} Jeffrey \\  $f(y)=(y-1) \log(y)$\end{tabular} & $\delta\hat{\beta}^{-1}\log^{-1}(x)$\\ \hline
  \begin{tabular}[c]{@{}c@{}} J-S \\  $f(y)=y\log(y) - (1+y) \log(\frac{1+y}{2})$\end{tabular} & \begin{tabular}[c]{@{}c@{}} $\ell$, where $\ell$ solves the equation \\
  $\ell \log(2) +   (1-\ell)\log(1-\ell)-(2-\ell)\log((2-\ell)/2)=\delta$, \\ 
  with $\delta \in (0, 2\log(2))$. \\ \end{tabular} \\ \hline
  \hhline{|=|=|}
  R{\'e}nyi &  \begin{tabular}[c]{@{}c@{}}$\Big(\dfrac{x}{\beta^*\sigma^*}\Big)^{-\beta^*}$ \\ where $\beta^*=\dfrac{\alpha-1}{\alpha}\hat{\beta}$, $\sigma^*=\dfrac{\alpha}{\alpha-1}(\exp((\alpha-1)\delta)-1)^{1/(\hat{\beta}(\alpha-1))}\hat{\sigma}$\end{tabular} \\ \hline
  \end{tabular}\vspace*{0.4cm}
  \caption{The tail behaviour of the worst-case distribution in $\mathcal{P}_f(\delta)$ for different $f$-divergences.}\label{table:examples}
  \end{table}
\end{example}

The form in Proposition~\ref{Pfdivergence} indicates that asymptotically, the tail behavior of the worst-case distribution in $\mathcal{P}_f(\delta)$ is influenced by the tail of the reference model. Moreover, Example~\ref{exDivergence} states a similar result as in \cite{BlanchetHeMurthy2020}: the qualitative information regarding the reference tail index is preserved under robustification when using $f$-divergence neighborhoods.

%===========================================================================================================

\section{Numerical examples}\label{Numerics}

In this section, we apply the bounds obtained in the previous sections to quantify the impact of model misspecification on the tail index. We illustrate our approach on a public data set of Danish insurance claims available in the \texttt{evir} package of the statistical software R.

This dataset represents a collection of large fire insurance claims $X_1, \dots, X_n$ in Denmark, from January 1980 until December 1990. It contains $n=2167$ observations, which can be considered as independent samples of the random variable $X$ representing claims over one million Danish Krone, in 1985 prices. The data was provided by Mette Rytgaard of Copenhagen Re (see \cite{Rytgaard1997}). A detailed analysis was performed in \cite{McNeil1997} with a focus on modelling the tail of the loss distribution. Figure~\ref{fig:analysis} shows an exploratory analysis of the data. In the left panel of this figure, the histogram  of log-transformed observations indicates a heavy-tailed distribution. The empirical mean residual life plot
$$\Big(u, \sum_{i=1}^{n}(X_i-u) 1\{ X_i > u \} /  \sum_{i=1}^{n}1\{ X_i > u \} \Big), \qquad u > 0,$$
is the empirical version of the conditional expectation $\mathbb E(X - u \mid X > u)$ as a function of the threshold $u$. If the distribution of the excesses of $X$ is a generalized Pareto distribution with scale $\sigma$ and shape $\xi$, then this plot should be concentrated around the line 
$$u \mapsto \frac{\sigma + \xi u}{1 - \xi},$$
\citep[e.g.,][Theroem 3.4.13]{EmbrechtsKluppelbergMikosch2013}. The right panel of Figure~\ref{fig:analysis} shows the mean excess plot for the fire insurance data. Above a threshold of $u=9.97$, which corresponds to the $95\%$ quantile, the plot is approximately linear and the approximation by a generalized Pareto distribution is justified.

\begin{figure}[H]
  \centering
  \includegraphics[width=.45\textwidth]{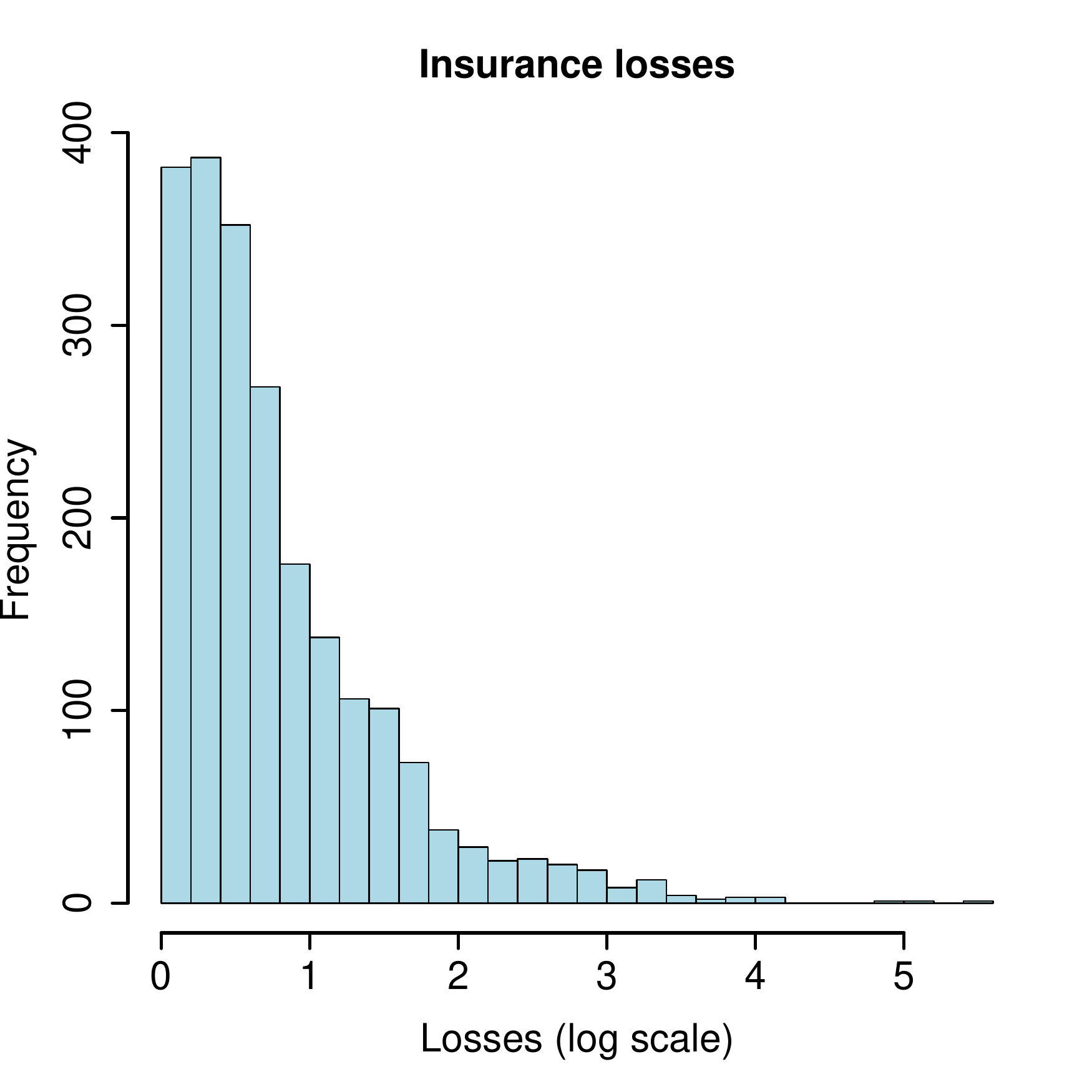} % 
  \includegraphics[width=.45\textwidth]{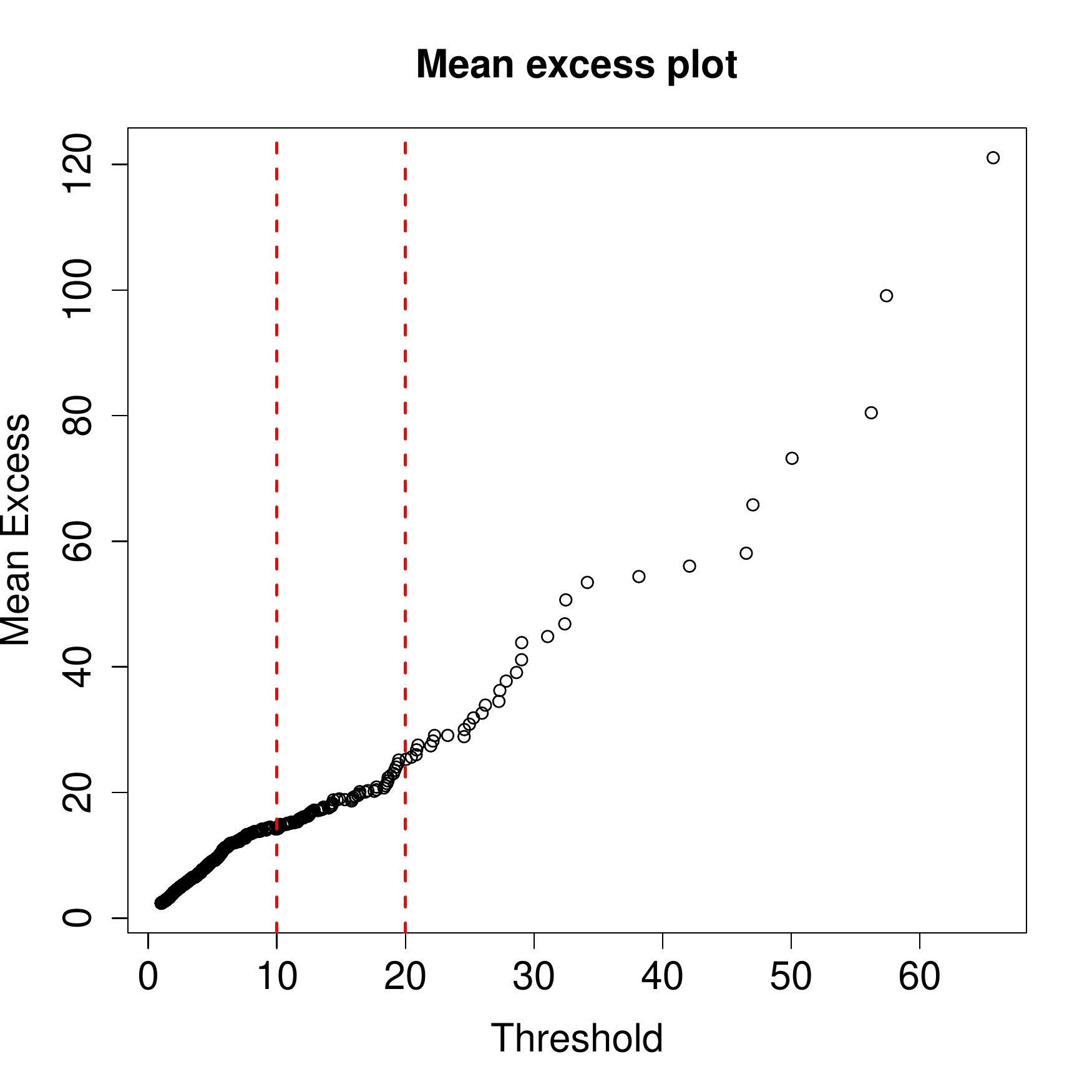}\vspace*{0.5cm}
  \caption{Exploratory analysis of Danish fire insurance data set. Left: histogram of log-transformed data. Right: mean excess plot: the mean excess function is approximately linear between threshold  u=9.97 and u=20 ( the two red dotted lines). }\label{fig:analysis}
\end{figure}

Using maximum likelihood estimation, we therefore fit a generalized Pareto distribution to the conditional distribution $X - u \mid X> u$ based on the exceedances $\{X_i : X_i > u\}$ of insurance claims above the $95 \%$ threshold $u = 9.97$. This results in a semi-parametric model for the tail of $X$, since 
\begin{align}
  \notag       
  P ( X > x) &= P( X > u) P( X - u > x-u \mid X>u)\\
  \label{eqFitted}  & \approx \hat{P}( X > u) \bigg(1+\dfrac{x-u}{\hat{\beta}\hat{\sigma}}\bigg)^{-\hat{\beta}}=: p_x, \quad x \geq u,
\end{align}
where $\hat{P}( X > u)$ is the empirical estimate of $P( X > u)$, and obtained the estimated scale $\hat{\sigma}=7.034$ and the estimated tail index $\hat{\beta}=2.03$ \citep[see also][]{McNeil1997}. The generalized Pareto distribution in~\eqref{eqFitted} is only an approximation of the tail of $X$, and for risk assessment we have to account for the model uncertainty induced by this modelling choice. From now on, the estimated generalized Pareto distribution will play the role of the reference model $\hat P$ around which different ambiguity sets are constructed. More precisely, we consider the worst-case tail $1 -  F^*_{\mathcal D}(x)$ in~\eqref{optim} for different divergence $\mathcal D$.

Based on our results in Sections~\ref{sWasserstein} and~\ref{sDivergence} on Wasserstein distance and $f$-divergences, respectively, we investigate two different approaches, described in the following.

\subsection{Pre-asymptotic analysis}

First, we consider a pre-asymptotic analysis where we do not rely on the approximations in Propositions~\ref{pPower} and~\ref{Pfdivergence}, but where we solve the optimization problem~\eqref{optim} explicitly for each $x>u$.
\begin{enumerate}
\item For the Wasserstein distance $\mathcal D = \WD_{\mathrm{d},1}$ for $\mathrm{d}$ as in Remark~\ref{rem1} for
some $s\geq 1$, according to the proof of Proposition~\ref{pPower}, the worst-case tail satisfies
\[
  1-F^*_{\WD}(x) = \hat P( U(x)^{1/s}, \infty), \quad  x > u,
\] 
where $U(x)$ is defined as the solution of the equation
\begin{equation}\label{eqUx}
  \delta = \int_{U(x)^{1/s}}^x (x^s-y^s)\,d\hat{F}(y),
\end{equation}
with $\hat F(y) = 1-\hat P(X >y)$. This equation can be solved numerically.
\item Similarly for the $f$-divergence $\mathcal D = D_f$, according to the proof of Proposition~\ref{Pfdivergence}, the worst-case tail satisfies
\[
  1-F^*_{f}(x) =  b_x \hat P( x, \infty), \quad  x > u,
\]
where $b_x\in (1,1/p_x)$ is defined as the solution of the equation
\begin{equation*}
  \delta = (1-p_x)f\bigg(\dfrac{1-p_xb_x}{1-p_x}\bigg) + p_xf(b_x),
\end{equation*}
with $p_x:=\hat{P}(x,\infty)$. Again, this equation can be solved numerically.
\end{enumerate}

First, we consider a pre-asymptotic analysis where we do not rely on the approximations in Propositions~\ref{pPower} and~\ref{Pfdivergence}, but where we solve the optimization problem~\eqref{optim} explicitly for each $x>u$.
\begin{enumerate}
\item For the Wasserstein distance $\mathcal D = \WD_{\mathrm{d},1}$ for $\mathrm{d}$ as in Remark~\ref{rem1} for
some $s\geq 1$, according to the proof of Proposition~\ref{pPower}, the worst-case tail satisfies
\[
  1-F^*_{\WD}(x) = \hat P( U(x)^{1/s}, \infty), \quad  x > u,
\] 
where $U(x)$ is defined as the solution of the equation
\begin{equation}\label{eqUx}
  \delta = \int_{U(x)^{1/s}}^x (x^s-y^s)\,d\hat{F}(y),
\end{equation}
with $\hat F(y) = 1-\hat P(X >y)$. This equation can be solved numerically.
\item Similarly for the $f$-divergence $\mathcal D = D_f$, according to the proof of Proposition~\ref{Pfdivergence}, the worst-case tail satisfies
\[
  1-F^*_{f}(x) =  b_x \hat P( x, \infty), \quad  x > u,
\]
where $b_x\in (1,1/p_x)$ is defined as the solution of the equation
\begin{equation*}
  \delta = (1-p_x)f\bigg(\dfrac{1-p_xb_x}{1-p_x}\bigg) + p_xf(b_x),
\end{equation*}
with $p_x:=\hat{P}(x,\infty)$. Again, this equation can be solved numerically.
\end{enumerate}

Before starting the procedure outlined above, we need to estimate the tolerance level $\delta$. In general, there are different ways to derive an empirical value of $\delta$. A data-driven way to determine the tolerance level involves computing the divergence between the reference model and some alternative model, inferred from a sample of data. The estimation procedure then focuses either on the density of models involved (\cite{EngelkeIvanovs2017}) or on the divergence itself (\cite{NguyenWainwrightJordan2010}, \cite{PoczosSchneider}). In the context of distributionally robust optimization of the mean-variance problem with Wasserstein ambiguity sets, \cite{BlanchetChenZhou2018} prove that the optimal choice of $\delta$ is of order $O(n^{-1})$, where $n$ is the number of returns in a time series data.

We aim to illustrate the differences between the tail of the reference model and the tails of worst-case models in the  Wasserstein and $f$-divergence ambiguity sets. We focus here on the Hellinger divergence, for which the parameters of the robust models are obtained in Example~\ref{exDivergence}. For the Wasserstein ambiguity set, we select a power $s<\hat{\beta}$ to construct the underlying distance $\mathrm{d}$ in Remark~\ref{rem1}. An estimate for the ambiguity radius $\delta$ is 
\[
  \delta_{\WD}:=\WD_{\mathrm{d},1}(\hat{F},\hat{F}_n), 
\]
where $\hat{F}$ is the fitted distribution, $\hat{F}_n$ is the empirical probability distribution and the Wasserstein distance $\WD_{\mathrm{d},1}$ is computed according to Remark~\ref{rem1}. For our example, the estimate $\delta_{\WD} = 3.2$ is used to solve numerically \eqref{eqUx} and thus obtain $U(x)$. In the case of Hellinger ambiguity set, we follow the methodology in \cite{BlanchetHeMurthy2020}. Let $(\hat{\beta}^{-1}-\varepsilon, \hat{\beta}^{-1}+\varepsilon)$ be the $95\%$ confidence interval for the shape parameter and choose the order $\alpha$ such that the reciprocal of the worst-case tail index (see Table~\ref{table:examples}) matches the upper end point of this interval, that is,
\begin{equation}\label{eqAlpha}
\hat{\beta}^{-1}\dfrac{\alpha}{\alpha-1} = \hat{\beta}^{-1} +\varepsilon.
\end{equation}
For our example, we have $\alpha = 2.86$. The $k$-nearest neighbor algorithm of \cite{PoczosSchneider} is then used to estimate the ambiguity radius $\delta_{\mathcal{H}}:= \mathcal{H}_\alpha(\hat{F}\Vert \hat{F}_n)$ between the fitted and the empirical distribution. In this case, we follow \cite{LallSharma1996} and choose $k$ of order $n^{1/2}$, where $n$ is the sample size. Using the estimation procedure in \cite{PoczosSchneider}, the resulting estimate is $\delta_{\mathcal{H}} =  0.01$. 

\begin{figure}[H]
  \begin{subfigure}{\linewidth}
  \includegraphics[width=.5\linewidth]{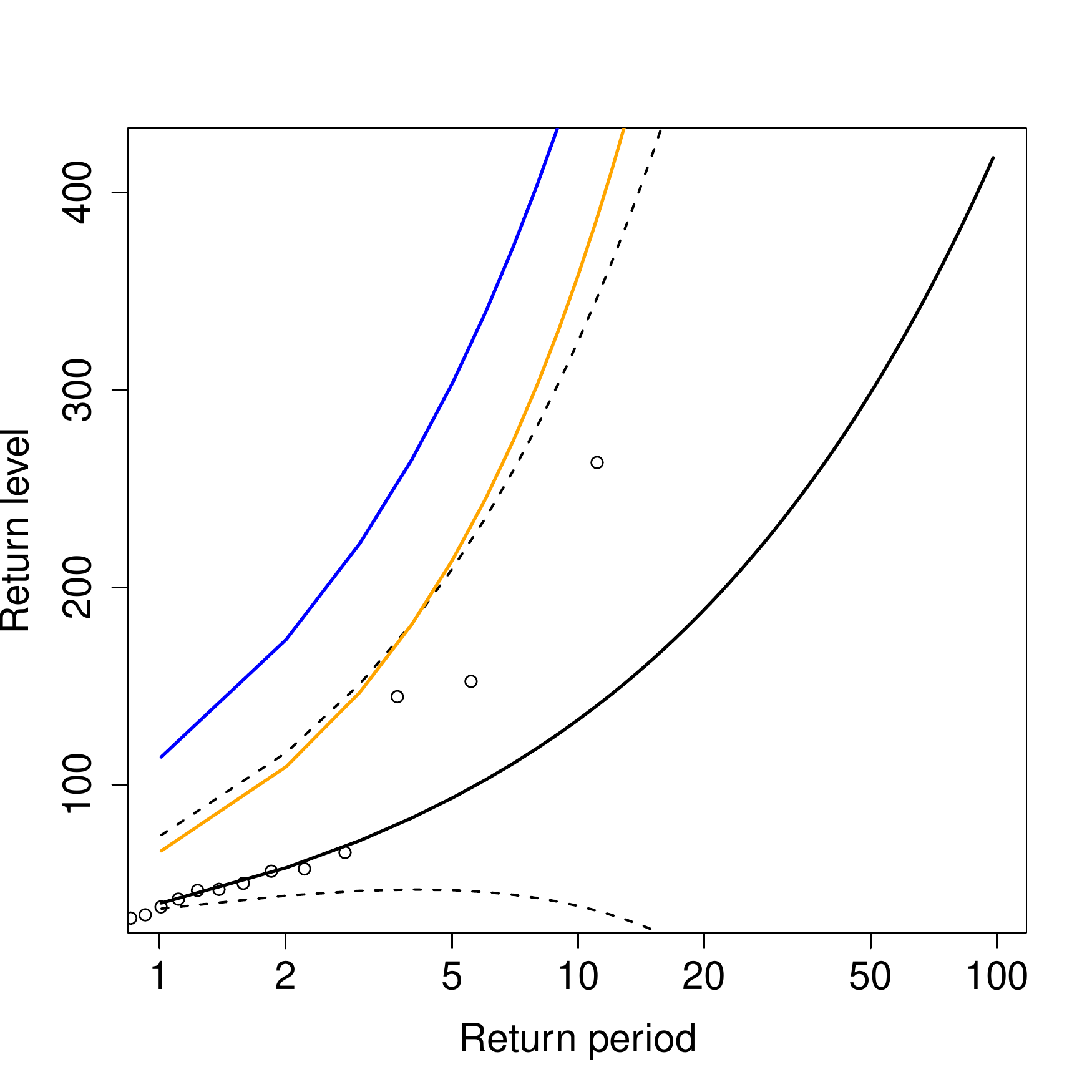}\hfill
  \includegraphics[width=.5\linewidth]{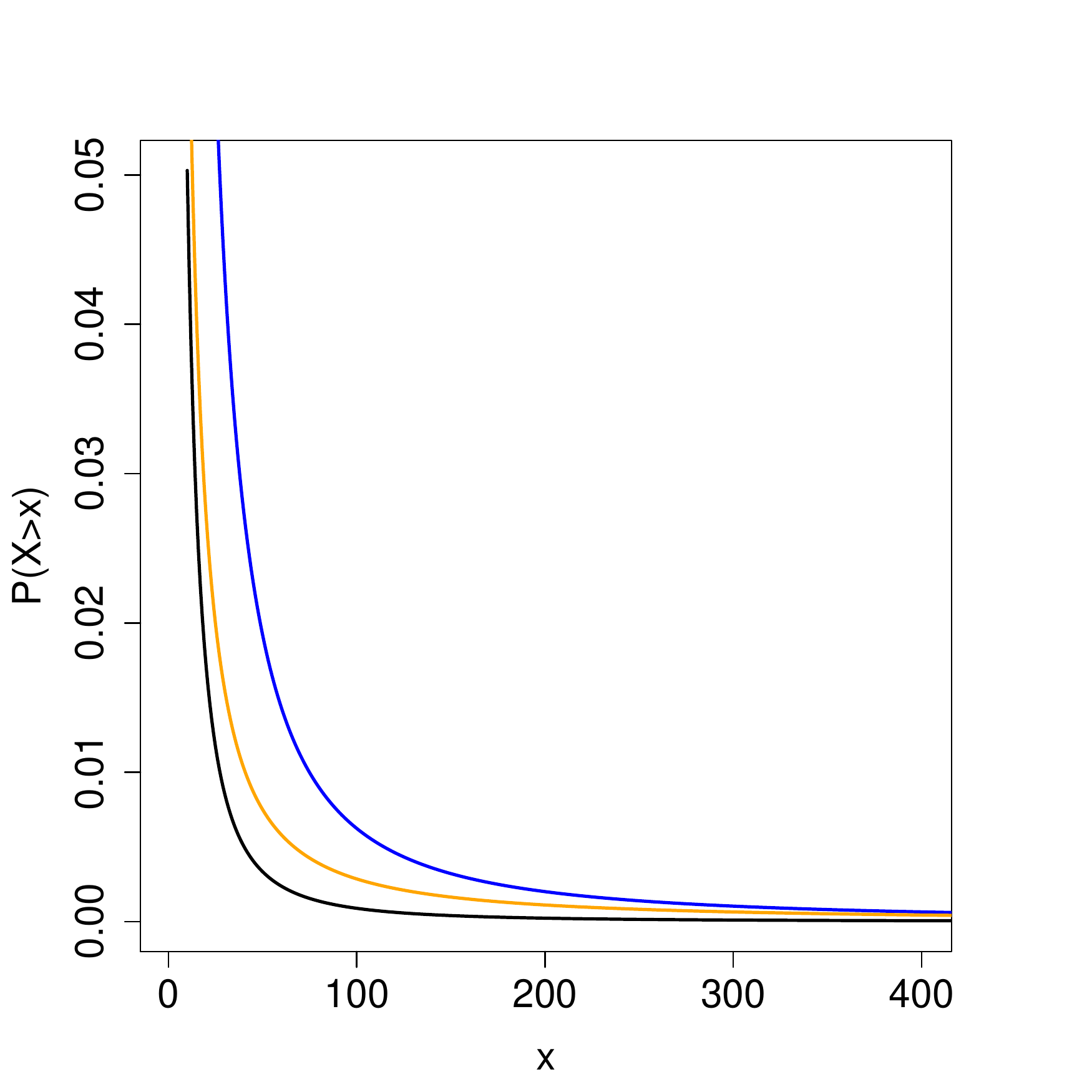}\hfill
  \end{subfigure}\vspace*{0.5cm}
  \caption{Left: the return levels (for return periods in years) of the reference model (fitted generalized Pareto distribution) at $95\%$ threshold (black) and corresponding $95\%$ confidence intervals (dashed lines); return levels of the worst-case model from Wasserstein ambiguity set with $s=1.5$ and $\delta_{\WD}= 3.2$ (blue) and from Hellinger ambiguity set with $\alpha=2.86$ and $\delta_{\mathcal{H}} = 0.01$ (orange). Right: the corresponding tail probabilities with the same color coding.}\label{fig:preasymptotic}
\end{figure}

The left-hand panel of Figure~\ref{fig:preasymptotic} shows the return levels (for return periods in years) of the fitted generalized Pareto distribution at $95\%$ threshold, the worst-case model in the Wasserstein ambiguity set with the power $s=1.5$ and estimated radius $\delta_{\WD}=3.2$, and the worst-case model in the Hellinger ambiguity set with the estimated order of divergence $\alpha=2.86$ and the estimated radius $\delta_{\mathcal{H}} = 0.01$. We observe that the return levels for $F_{\WD}^*$ are higher than the return levels for both the reference model $\hat{F}$ and $F_{\mathcal{H}_\alpha}^*$. This is explained by the fact that the worst-case distribution in the Wasserstein ambiguity set has a heavier tail, compared to the tail of the reference distribution and the worst-case distribution in the Hellinger ambiguity set; see right-hand panel of Figure~\ref{fig:preasymptotic}. It is interesting to note that both worst-case tails are only slightly more conservative than the confidence intervals. That indicates that there is not a large error due to model misspecification by using the generalized Pareto distribution in this data set.

As opposed to the order $\alpha$ in the Hellinger divergence, the order $s$ of the Wasserstein distance cannot be chosen in a data-driven way. Therefore we conducted a further study assessing the impact of different values of $s$ on the worst-case Wasserstein bounds (not shown here). It turns out that the bounds are fairly stable across different values for $s$, and it is more important to estimate the corresponding $\delta$ values accurately.

\subsection{Asymptotic analysis}

In the previous section, we have compared the estimated model $\hat{F}$ with pre-asymptotic worst-case models in different neighborhoods around $\hat{F}$. However, we can avoid the numerical optimization and rely on the asymptotic approximations for the worst-case tails derived in Sections~\ref{sWasserstein} and~\ref{sDivergence}.

\begin{enumerate}
\item For the Wasserstein distance $\mathcal D = \WD_{\mathrm{d},1}$ for $\mathrm{d}$ as in Remark~\ref{rem1} for some $s\geq 1$, according to Proposition~\ref{pPower}, the worst-case tail satisfies
\[
  1-F^*_{\WD}(x) \sim \delta x^{-s}, \quad \text{ as } x \to \infty.
\] 
\item Similarly for the $f$-divergence $\mathcal D = D_f$, according to Proposition~\ref{Pfdivergence}, the worst-case tail satisfies
\[
  \overline{F}^*_f(x)\sim f^\leftarrow(\delta/p_x)p_x, \quad \text{ as } x\to \infty.
\]
\end{enumerate}
Similarly as in the pre-asymptotic setting, we fix as reference model $\hat{F}$ the semi-parametric model fit in~\eqref{eqFitted}. Figure~\ref{fig:asymptotic} shows the return levels (left) computed from this reference model together with $95\%$ confidence intervals, the return levels from the robust Wasserstein model with the power of distortion $s=1.5$ and the return levels from the robust Hellinger model at estimated order $\alpha=2.86$. The parameters are chosen according to the procedure outlined in the previous section. The computations are performed for ambiguity radii $\delta_{\WD}=3.2$ and $\delta_{\mathcal{H}}=0.01$. The right panel illustrates the tail decays for $\hat{F}$, $F_{\WD}^*$ and $F_{\mathcal{H}_\alpha}^*$.

\begin{figure}[H]\vspace*{-1cm}
  \begin{subfigure}{\linewidth}
  \includegraphics[width=.5\linewidth]{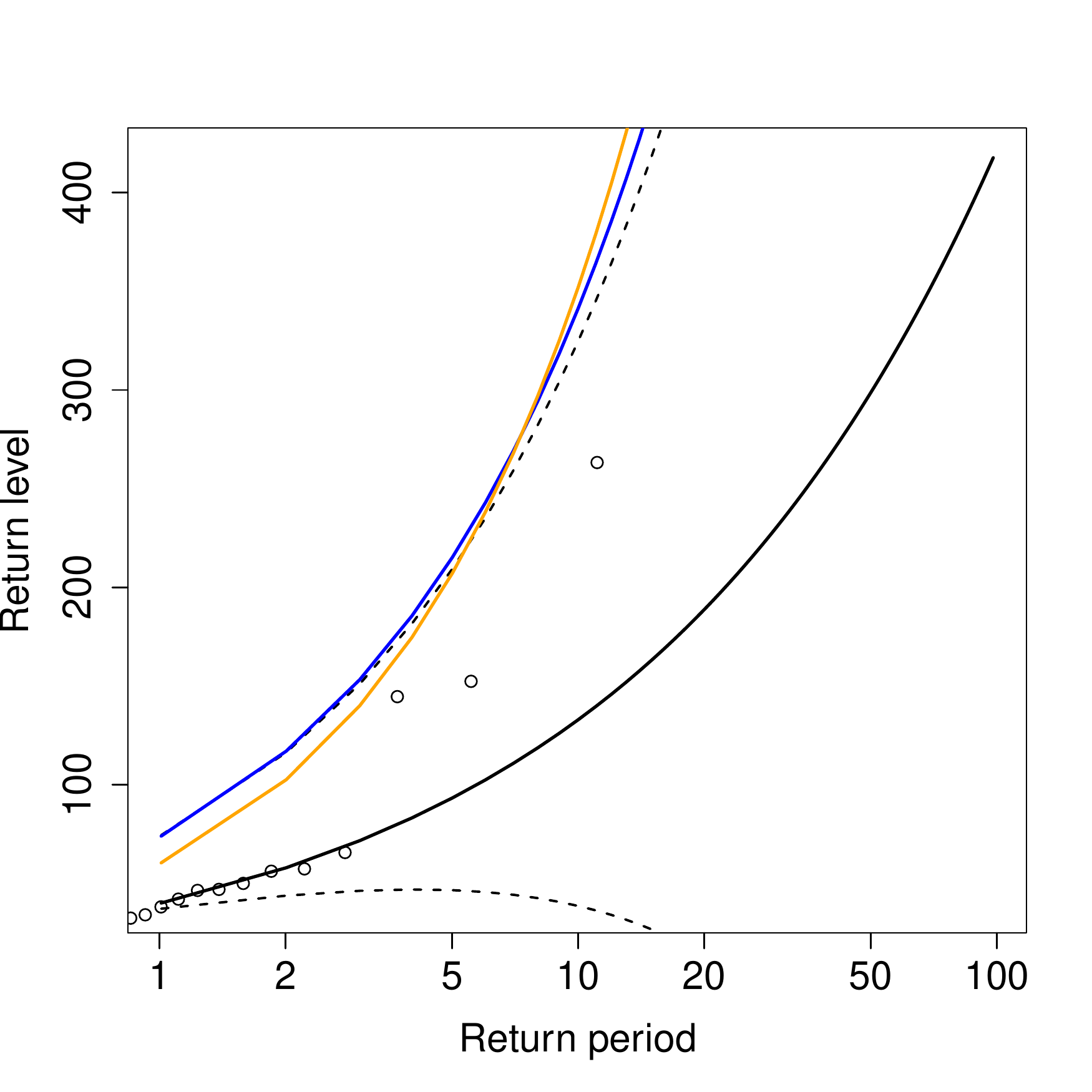}\hfill
  \includegraphics[width=.5\linewidth]{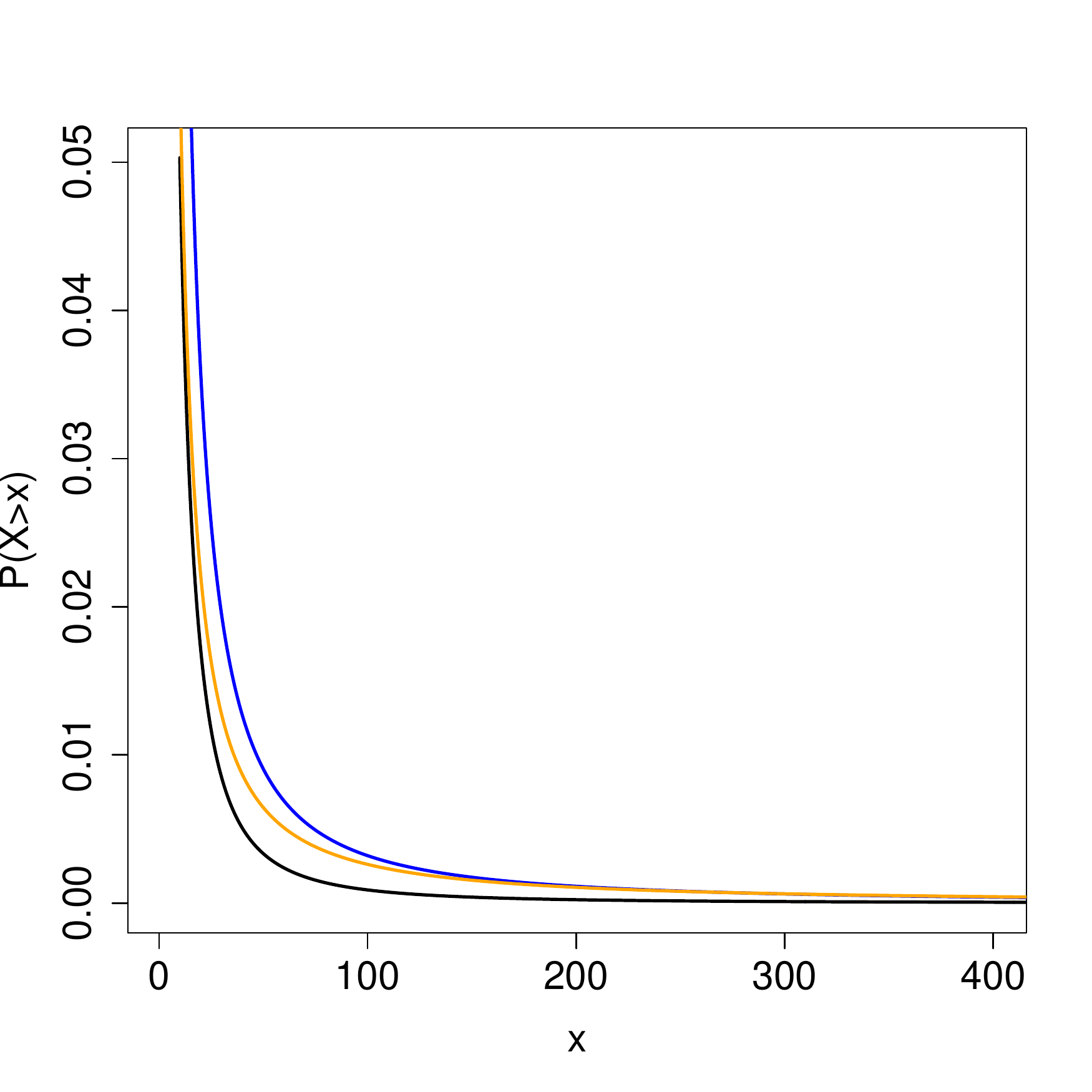}\hfill
  \end{subfigure}\vspace*{0.3cm}
  \caption{The return levels (left) and the tail decays (right) of the reference model at $95\%$ threshold (black), $95\%$ CI (dashed line),  worst-case model from Wasserstein set (blue) and Hellinger (orange) for $s=1.5$ and $\delta_{\WD} = 3.2$, and $\alpha=2.86$ and $\delta_{\mathcal{H}}=0.01$, respectively.}\label{fig:asymptotic}
\end{figure}

The figures show that the asymptotic robust bounds are even closer to the $95\%$ confidence intervals of the references model. This is another indication that the class of generalized Pareto distributions is a good model for the data above the threshold $u$. The fact that we have chosen a ``good'' threshold $u$ based on the mean excess plot in Figure~\ref{fig:analysis}, has certainly helped to minimize model uncertainty in the first place. We observe, however, that especially far in the tail, the tails of the worst-case distributions $F_{\WD}^*$ and $F_{\mathcal{H}_\alpha}^*$ are heavier than the tail of $\hat{F}$. 

%===========================================================================================================

\bibliographystyle{Chicago}
\bibliography{bibliography}

%===========================================================================================================

\begin{appendices}

\section{Technical result to complete the proof of Remark~\ref{rAssump}}\label{appendixA}

\begin{proposition}\normalfont
  Let $f\in\mathcal{C}^1(0,\infty)$ be a strictly convex divergence function such that $f\in\mathcal{R}_{\rho}$ with $\rho \geq 0$. Define the divergence $\tilde{f}:(0,\infty)\rightarrow\mathbb{R}$, $\tilde{f}(y):=f(y)-f'(1)(y-1)$. Then $\tilde{f}$ satisfies the following properties:
  \begin{enumerate}[label=(\alph*)]
    \item $\tilde{f}'(1)=0$;
    \item $\tilde{f}\geq 0$ on $[1,\infty)$;
    \item $\tilde{f}$ is increasing on $[1,\infty)$;
    \item there exists some $b\in\mathbb{R}_+$ such that the divergence $f:[0,\infty)\rightarrow \mathbb{R}$, $f_1(y)=\tilde{f}(y) + b(y-1)$ is such that $f_1\in\mathcal{R}_{\geq 1}$ and $D_{\tilde{f}}(P\Vert Q) = D_{f_1}(P\Vert Q)$, for all $P\ll Q$.
  \end{enumerate}
\end{proposition}
\begin{proof}
  \begin{enumerate}[label=(\alph*)]
    \item  Clear by definition of $\tilde{f}$. 
    \item The first order condition for differentiable convex functions yields that $\tilde{f}(y) \geq \tilde{f}'(1)(y-1)\geq 0$ for all $y\geq 1$.
    \item Since $f$ is strictly convex, then $\tilde{f}$ is strictly convex and thus $\tilde{f}'$ is increasing. For $1\leq y_1<y_2$, the first order condition states that 
    \[
    \tilde{f}(y_2)-\tilde{f}(y_1)\geq \tilde{f}'(y_1)(y_2-y_1)\geq \tilde{f}'(1)(y_2-y_1)= 0,
    \]
    where the last inequality follows from the monotonicity of $\tilde{f}'$.
    \item Assume that $f\in\mathcal{R}_\rho$ is regularly varying with index $\rho\in [0,1)$. If $f'(1)<0$, then $\tilde{f}$ is regularly varying function of index $\max\lbrace \rho,1\rbrace =1$ and, in this case, $b=0$. Else, if $f'(1)>0$, let $b>f'(1)$ and define $f_1(y)=\tilde{f}(y)+b(y-1)= f(y)+ (b-f'(1))(y-1)$. Also in this case, $f_1$ is regularly varying of index $\max\lbrace \rho,1\rbrace=1$.  
  \end{enumerate}
\end{proof}

%===========================================================================================================

\section{Technical result to complete the proof of Proposition~\ref{Pfdivergence}}\label{appendixB}

Lemma \ref{lPxbx} below ensures the convergence needed to finish the proof of Proposition \ref{Pfdivergence}.

\begin{lemma} \label{lPxbx}\normalfont
  Let $\hat{P}$ be a reference model and let $x_0\geq 0$ such that $p_x:=\hat{P}(x,\infty)>0$ for all $x \geq x_0$ and $\lim_{x\to \infty} p_x = 0$. Let $f:(0,\infty)\rightarrow\mathbb{R}$ be a divergence satisfying Assumption~\ref{aF1}. Suppose that for all $x \geq x_0$ there exists $b_x\in (1,1/p_x)$ such that
\begin{equation}\label{eqBx}
p_xf(b_x)+(1-p_x)f\bigg(\dfrac{1-p_xb_x}{1-p_x}\bigg)=\delta,
\end{equation}
where $\delta<f^*(0)+f(0)$. Then it holds that
\begin{itemize}
\item[(i)]
  $\lim_{x\rightarrow\infty}p_xb_x=0$ if $f^*(0) = \infty$;
\item[(ii)]
  $\lim_{x\rightarrow\infty}p_xb_x= \ell \in (0,1]$ if $f^*(0) \in (0,\infty)$.
\end{itemize}
\end{lemma} 
\begin{proof}
For (i), assume by contrary that $\lim\sup_{x\rightarrow\infty}p_xb_x= \ell \in(0,1]$. By \eqref{eqBx} it then follows that  $\lim\sup_x p_xf(b_x)<\infty$. Moreover, since $p_x \to 0$, we have that $b_x \to \infty$. Consequently,
\[
\lim_{x\rightarrow\infty} p_xf(b_x) =  \lim_{x\rightarrow\infty} p_xb_x f(b_x)/ b_x = \ell \lim_{x\rightarrow\infty} f(b_x)/ b_x =\infty,
\]
since $\lim_{y\to\infty} f(y)/y =  f^*(0) = \infty$.
This contradicts $\lim\sup_x p_xf(b_x)<\infty$ and thus  $\lim_{x\rightarrow\infty}p_xb_x = 0$.

For (ii), suppose that $p_x b_x$ has the accumulation point $\ell = 0$ for the sequence $(x_k)$. The limit of the left-hand side of \eqref{eqBx} is equal to the limit $\lim_{k \to \infty} p_{x_k}f(b_{x_k})$ since $f(1) = 0$. Thus, since $p_{x_k} \to 0$, for \eqref{eqBx} to hold it is necessary that $f(b_{x_k}) \to \infty$, and since $b_{x_k} > 1$ that means that $b_{x_k} \to \infty$. Since then $f(b_{x_k}) \sim f^*(0) b_{x_k}$ we have $\lim_{k \to \infty} p_{x_k}f(b_{x_k}) = \lim_{k \to \infty} p_{x_k}b_{x_k} f^*(0) = 0$. This contradicts the fact that $b_x$ solves  \eqref{eqBx}. Therefore $0$ is not an accumulation point of $p_x b_x$.

Suppose now that $p_x b_x$ has the accumulation point $\ell \in(0,1]$ for the sequence $(x_k)$. Then $\lim_{k \to \infty} b_{x_k} = \infty$. For large $k$, the limit of the left-hand side of \eqref{eqBx} for this sequence as $k\to \infty$ is $g(\ell) := \ell f^*(0) + f(1-\ell)$ since $f(y) \sim f^*(0) y$. Note that $g'(\ell) = f^*(0) - f'(1 - \ell) > f^*(0) - f'(1) = f^*(0) > 0$ for $\ell \in (0,1]$. The function $g$ is therefore strictly increasing on $(0,1]$ with $g(0) = 0$ and $g(1) = f^*(0) + f(0)$. Since $\delta < f^*(0) + f(0)$, there exists a unique $\ell\in (0,1]$ with $g(\ell) = \delta$. Therefore, all accumulation points of $p_x b_x$ are the same and  $\lim_{x\rightarrow\infty}p_xb_x= \ell \in (0,1]$. 

\end{proof}

%===========================================================================================================

\section{Worst-case distribution for $f$-divergence ambiguity sets}\label{appendixC}

Let the reference model $\hat{P}$ be a generalized extreme value distribution with tail index $\hat{\beta}>0$ such that 
\[ 
  p_x=\hat{P}(x,\infty)\approx \bigg(\dfrac{x}{\hat{\beta}\hat{\sigma}}\bigg)^{-\hat{\beta}}, \quad  x\rightarrow\infty.
\]

\begin{enumerate}
\item \textbf{Kullback--Leibler divergence.} Since $f(y)=y\log(y)$, then $f^{-1}(t)=\dfrac{t}{W(t)}$, where $W(t)$ is the Lambert function, for which it holds that $W(t)=\log(t)-\log(\log(t))+o(1)$ (see \cite{Hassani2005}). It follows that 
\[
\overline{F}^*_{\KL}(x)\sim f^{-1}(\delta p_x^{-1})p_x = \dfrac{\delta}{\log(\delta p_x^{-1})-\log(\log(\delta p_x^{-1}))+o(1)}\sim\dfrac{\delta}{\log(\delta p_x^{-1})}\approx \dfrac{\delta}{\hat{\beta}\log(x)},
\]
as $x\rightarrow\infty$.
%===========================================================================================================

\item \textbf{Hellinger divergence.} The worst-case tail distribution $\overline{F}^*_{\mathcal{H}_\alpha}$ in a Hellinger divergence ball is of the form:
\begin{equation*}
\begin{split}
\overline{F}^*_{\mathcal{H}_\alpha} & \sim f^{-1}(\delta p_x^{-1})p_x = \big(1+\delta(\alpha-1)p_x^{-1})\big)^{1/\alpha}p_x = \big[\delta(\alpha-1)p_x^{-1}\big]^{1/\alpha}\big(p_x(\delta(\alpha-1))^{-1}+1\big)^{1/\alpha}p_x \\ 
& \approx p_x^{1-1/\alpha}\big[\delta(\alpha-1)\big]^{1/\alpha} \approx(\delta(\alpha-1))^{1/\alpha} \bigg(\dfrac{x}{\hat{\beta}\hat{\sigma}}\bigg)^{-\hat{\beta}(1-1/\alpha)}, \, x\rightarrow\infty.
\end{split}
\end{equation*}
where we use $(1+y)^a\approx 1+ay$, for $\vert y\vert <1$ and $\vert ay\vert \ll 1$.  Hence $\overline{F}^*_{\mathcal{H}_\alpha}(x)\approx \bigg(\dfrac{x}{\beta^*\sigma^*}\bigg)^{-\beta^*}$ with $\beta^*=\dfrac{\alpha-1}{\alpha}\hat{\beta}$ and $\sigma^*= \alpha(\alpha-1)^{\frac{1}{\hat{\beta}(\alpha-1)}-1}\delta^{\frac{1}{\hat{\beta}(\alpha-1)}}\hat{\sigma}$.

In particular, for $\alpha=2$, the worst-case tail distribution in a $\chi^2$-divergence case has a tail index $\beta^*=\hat{\beta}/2$ and a scale $\sigma^*=2\delta^{1/\hat{\beta}}\sigma$.

%===========================================================================================================

\item \textbf{Triangle discrimination.} For $f(y)=\dfrac{(y-1)^2}{y+1}$, observe that $\lim_{y\rightarrow\infty} f(y)/y=1$. Thus the corresponding worst-case distribution is $\overline{F}^*_{\Delta}(x)\sim \ell$, where $\ell$ solves the equation 
\[
  \ell + \dfrac{\ell^2}{2-\ell}=\delta.
\]
We obtain that $\ell= \dfrac{2\delta}{\delta + 2}\leq 1$, since $\delta\leq f^*(0)+f(0)=2$.

%===========================================================================================================

\item \textbf{Jeffrey divergence.} Let us denote by $f_1(y):=(y-1)\log(y)$ and $f_2(y):=y\log(y)$ the functions corresponding to Jeffrey- and $\KL$- divergence, respectively. Since both $f_1,f_2$ are increasing and positive on $[1,\infty)$ and moreover, $f_2\in\mathcal{R}_{-1}$, it follows from Theorem A in \cite{DjurvcicTorgavsev2007} that if $f_1(y)\sim f_2(y)$, as $y\rightarrow\infty$, then $f_1^{-1}(t)\sim f_2^{-1}(t)$, as $t\rightarrow\infty$. The worst-case tail in the Jeffrey divergence ambiguity set becomes $\overline{F}^*_{f_1}(x)\sim f_1^{-1}(\delta p_x^{-1})p_x\sim f_2^{-1}(\delta p_x^{-1})p_x\sim\delta\hat{\beta}^{-1}\log^{-1}(x)$.

%===========================================================================================================

\item \textbf{Jensen--Shannon divergence.} We observe that $f(y)\sim y\log(2)$, as $y\rightarrow\infty$ and $y\log(2)\in\mathcal{R}_1$, where $f(y)=y\log(y)-(1+y)\log((y+1)/2)$. The corresponding worst-case distribution is $\overline{F}^*_f(x)\sim \ell$, where $\ell$ solves the equation 
\[
  \ell \log(2) +   (1-\ell)\log(1-\ell)-(2-\ell)\log((2-\ell)/2)=\delta, 
\]
where $\delta\leq f^*(0)+f(0)=2\log(2)$.

%===========================================================================================================

\item \textbf{R{\'e}nyi divergence: } Problem \eqref{optim} has been studied before in \cite{BlanchetHeMurthy2020} when the measure of discrepancy is the R{\'e}nyi divergence. As the R{\'e}nyi divergence is connected to the  Hellinger divergence, we can recover the tail index of the worst-case distribution obtained in \cite{BlanchetHeMurthy2020}.

\begin{definition}\normalfont
Consider two probability measures $P$ and $Q$ on $S$ such that $P$ is absolutely continuous with respect to $Q$. For any $\alpha>1$, the R{\'e}nyi divergence of degree $\alpha$ is defined as
\[
D_\alpha(P\Vert Q):=\dfrac{1}{\alpha-1}\log\mathbb{E}_Q\bigg[\bigg(\dfrac{dP}{dQ}\bigg)^\alpha\bigg].
\]
\end{definition}

The R{\'e}nyi divergence of order $\alpha\geq 1$ is a one-to-one transformation of the Hellinger divergence:
\[
D_\alpha(P\Vert \hat{P})=\dfrac{1}{\alpha-1}\log(1+(\alpha-1)\mathcal{H}_\alpha(P\Vert \hat{P})). 
\]
Then $D_\alpha(P\Vert \hat{P})=\delta$ is equivalent to $\mathcal{H}_\alpha(P\Vert \hat{P})=\dfrac{\exp((\alpha-1)\delta)-1}{\alpha-1}=\overline{\delta}$.  As $p_x\approx\bigg(\dfrac{x}{\hat{\beta}\hat{\sigma}}\bigg)^{-\hat{\beta}}$, then the worst-case tail distribution has a tail index $\beta^*=\dfrac{\alpha-1}{\alpha}\hat{\beta}$ and a scale parameter $\sigma^*$ of the form 
\begin{equation*}
\sigma^* = \alpha(\alpha-1)^{1/(\hat{\beta}(\alpha-1)-1)}\overline{\delta}\,^{1/(\hat{\beta}(\alpha-1))}\hat{\sigma} = \dfrac{\alpha}{\alpha-1}\big(\exp((\alpha-1)\delta)-1\big)^{1/(\hat{\beta}(\alpha-1))}\hat{\sigma}.
\end{equation*}
\end{enumerate}

\end{appendices}
\end{document}